\documentclass[12pt]{article}

\usepackage[utf8]{inputenc}
\usepackage[T1]{fontenc}
\usepackage{lmodern}
\usepackage[english]{babel}
\usepackage[normalem]{ulem}
\usepackage[usenames]{color}
\usepackage{nccfoots}
\usepackage{appendix}
\usepackage{hyphenat, cite}
\usepackage{hyperref, graphicx}
\usepackage{enumerate, tikz}
\usepackage{exscale, booktabs}
\usepackage[a4paper, top=2cm, bottom=2cm, left=2cm, right=2cm]{geometry}
\usepackage[nodisplayskipstretch]{setspace}
\usepackage{amsmath, amssymb}
\usepackage{amsfonts, bbm}
\usepackage{makeidx, latexsym}
\usepackage{amscd,mathrsfs}
\usepackage{amsthm}
\numberwithin{equation}{section}
\theoremstyle{plain}
\newtheorem{Lemma}{Lemma}[section]
\newtheorem{Proposition}[Lemma]{Proposition}
\newtheorem{Theorem}[Lemma]{Theorem}
\newtheorem{Corollary}[Lemma]{Corollary}

\theoremstyle{definition}
\newtheorem{Definition}[Lemma]{Definition}
\newtheorem{Example}[Lemma]{Example}

\newtheorem{Remark}[Lemma]{Remark}

\allowdisplaybreaks

\let\OLDthebibliography\thebibliography
\renewcommand\thebibliography[1]{
  \OLDthebibliography{#1}
  \setlength{\parskip}{0pt}
  \setlength{\itemsep}{2pt }}
\begin{document}
\title{\textbf{Mild and viscosity solutions to\\ semilinear parabolic path-dependent PDEs}}
\author{Alexander Kalinin\footnote{Department of Mathematics, University of Mannheim, Germany. Email: {\tt kalinin@math.uni-mannheim.de}} \and Alexander Schied\setcounter{footnote}{6}\footnote{Department of Statistics and Actuarial Science, University of Waterloo and Department of Mathematics, University of Mannheim. Email: {\tt alex.schied@gmail.com}\newline The authors gratefully acknowledge support by Deutsche Forschungsgemeinschaft (DFG) through Research Grants SCHI/3-1 and SCHI/3-2.}
\date{\small First version: November 24, 2016\\
\small This version:  November 14, 2018}}
\maketitle

\begin{abstract}
We study and compare two concepts for weak solutions to semilinear para\-bo\-lic path-dependent partial differential equations (PPDEs). The first is that of mild solutions as it appears, e.g., in the log-Laplace functionals of historical superprocesses. The aim of this paper is to show that mild solutions are also solutions in a viscosity sense. This result is motivated by the fact that mild solutions can provide value functions and optimal strategies for problems of stochastic optimal control. Since unique mild solutions exist under weak conditions, we obtain as a corollary a general existence result for viscosity solutions to semilinear parabolic PPDEs.
\end{abstract}


\section{Introduction}

The introduction of horizontal and vertical derivatives of non-anticipative functionals on path spaces by Dupire \cite{Dupire} and Cont and Fourni\'e \cite{ItoFormula} facilitated the formulation of a new class of \emph{path-dependent partial differential equations} (PPDEs). In relevant publications such as   Peng \cite{PengViscosity,PengICM}, Peng and Wang \cite{Peng2011}, Ji and Yang \cite{Ji2013}, Ekren, Keller, Touzi, and Zhang \cite{PPDEViscosity}, and Henri-Labordere, Tan, and Touzi \cite{HenryLabordereTanTouzi}, the most common approach to construct classical or viscosity solutions to PPDEs is to use backward stochastic differential equations (BSDEs). In this paper, we  propose an alternative approach in the case of a \emph{semilinear parabolic} PPDE,
\begin{equation}\label{TVP}
\left\{
\begin{array}{rll}
(\partial_{t} + \mathscr{L})(u)(t,x) &= f(t,x,u(t,x)) &\text{for $(t,x)\in [0,T)\times C([0,T],\mathbb{R}^{d})$,}\\
u(T,x) &= g(x) &\text{for $x\in C([0,T],\mathbb{R}^{d})$.}
\end{array}
\right.
\end{equation}
Here, $T> 0$, $d\in\mathbb{N}$, $\partial_{t}$ is the horizontal derivative, and $\mathscr{L}$ is a linear second-order differential operator of the form
\[
\mathscr{L}= \frac{1}{2}\sum_{i,j=1}^{d} a_{i,j}(t,x) \partial_{x_{i}}\partial_{x_{j}} + \sum_{i=1}^{d}b_{i}(t,x) \partial_{x_{i}},\quad (t,x)\in [0,T)\times C([0,T],\mathbb{R}^{d}),
\]
for the partial vertical derivatives $\partial_{x_i}$ and non-anticipative Borel measurable functions $a_{i,j}$ and $b_i$ such that the matrix $(a_{i,j}(t,x))$ is positive definite for all $(t,x)\in [0,T)\times C([0,T],\mathbb{R}^{d})$. Moreover, $f$ is a non-anticipative measurable function on $[0,T)\times C([0,T],\mathbb{R}^{d})\times D$, where $D\subset\mathbb{R}$ is a non-degenerate interval, and $g:C([0,T],\mathbb{R}^{d})\rightarrow D$ is Borel measurable.

Our starting point is the observation that for $D=\mathbb{R}_{+}$ a well-studied example of solutions to such a semilinear parabolic PPDE is provided by the log-Laplace functionals of a historical superprocess in the sense of Dawson and Perkins \cite{HistProcessesPerkins} and Dynkin \cite{PathProcesses}. Indeed, in the special case of a historical Brownian motion, these log-Laplace functionals are characterized by functions $u(r,x)$ that solve a Markovian integral equation of the form
\begin{equation}\label{mild solution integral eq intro}
u(r,x)=E_{r,x}[g(X^{T})]-E_{r,x}\bigg[\int_{r}^{T}f(u(s,X^{s}))\,ds\bigg]
\end{equation}
for all $(r,x)\in [0,T]\times C([0,T],\mathbb{R}^{d})$, where $f$ is, e.g., of the form $f(z)=z^p$ for all $z\in\mathbb{R}_{+}$ and some $p\in(1,2]$ and where under the probability measure $P_{r,x}$ the process $X$ has the law of a Brownian motion started at $x(r)$ at time $r$ and satisfies $X_s=x(s)$ for $s\in [0,r]$. The notation $X^s$ refers as usual to the process $X$ stopped at time $s$. It is then easy to see that $u$ corresponds to a mild solution to \eqref{TVP} if we let $a_{i,j}=\delta_{i,j}$ and $b_i=0$ for all $i,j\in\{1,\dots,d\}$. Mild solutions have the advantage that their existence can be established via standard methods such as Picard iteration and non-extendibility arguments under much less restrictive conditions than classical or BSDE solutions.  A general existence result is proved by the first author in the companion paper \cite{Kalinin}[Theorem 2.11], and we will state it here without proof as Theorem~\ref{mild sol existence thm}. 

Moreover, it was observed by the second author in \cite{SchiedFuel} that Laplace functionals of historical superprocesses yield value functions and optimal strategies for certain problems of optimal stochastic control arising in mathematical finance. Under rather mild conditions, the same arguments as in the proof of \cite[Theorem 2.8]{SchiedFuel} can actually be applied to any mild solution $u$ to \eqref{TVP} if $D=\mathbb{R}_{+}$ and $f(t,x,z)=\alpha(t,x)z^p$ for some $p>1$. Using these arguments, we show that $|z|^pu(r,x)$ is the value function of a certain stochastic control problem. A standard heuristic asserts that value functions should be viscosity solutions to a certain Hamilton--Jacobi--Bellman (HJB) equation. Following the reasoning in \cite[Section 1.1]{SchiedFuel}, one easily finds that this HJB equation should be indeed equivalent to our PPDE \eqref{TVP}. Therefore, it is a natural question whether our mild solution $u$ is also a solution in the viscosity sense. This is the main question we address in this paper. 
 Our main results, Theorems~\ref{Main Result 1} and~\ref{Main Result 2}, answer this question affirmatively under rather weak assumptions. In particular, we do not require that $u$ is continuous but only need a weaker notion of right-continuity. In the affine case, even right-continuity can be dropped. We observe moreover that the spaces of test functions used in \cite{PPDEViscosity} can be increased in our case, thus yielding a solution concept that is stronger than the one in \cite{PPDEViscosity} (compare also \cite[Remarks 3.5 and 3.8]{PPDEViscosity}). As a corollary to 
Theorems~\ref{mild sol existence thm}, ~\ref{Main Result 1}, and~\ref{Main Result 2} we obtain a general existence result for viscosity solutions to \eqref{TVP}. After the authors finished a previous version of this paper, they became aware of \cite{CossoEtal2018}, where results were obtained that partially precede our Theorem \ref{Main Result 1}.  

\medskip

This paper is organized as follows. In Section~\ref{prelim results section},
 we first recall some preliminaries. Specifically, in Section~\ref{Measurability section}, we discuss notation and topologies on path spaces and we propose a slight modification of the setup used, e.g.,  in \cite{ItoFormula,PPDEViscosity}. In Section~\ref{path space calculus section}, we recall from  \cite{Dupire,ItoFormula}  the definitions of horizontal and vertical derivatives on path spaces. In Section~\ref{TVP Section}, we give a precise formulation of the terminal value problem \eqref{TVP}, and in Section~\ref{mild solution section} we discuss the corresponding mild solutions and how they relate to our problem in stochastic optimal control. Section~\ref{Test functions for viscosity solutions} introduces two notions of viscosity solutions. Our main results are stated in Section~\ref{Main results Section}. The relations between the two notions of viscosity solutions presented here and the one in \cite{PPDEViscosity} are discussed in Section~\ref{Relations between the notions of viscosity solutions}. All proofs are deferred to Section~\ref{Derivation of the results}.

\section{Preliminaries and main results}\label{prelim results section}

Throughout the paper, let $T > 0$, $d\in\mathbb{N}$, and $|\cdot|$ be the Euclidean norm on $\mathbb{R}^{d}$.

\subsection{Measurability and right-continuity on path spaces}\label{Measurability section}

In the sequel, we let $\widetilde{S}$ denote the linear space of all $\mathbb{R}^{d}$-valued c\`{a}dl\`{a}g maps on the interval $[0,T]$ and set $S:=C([0,T],\mathbb{R}^{d})$. We work with the canonical process $\widetilde{\xi}:[0,T]\times\widetilde{S}\rightarrow\mathbb{R}^{d}$, $\widetilde{\xi}_{t}(x):= x(t)$ and its restriction $\xi$ to $[0,T]\times S$. By $(\widetilde{\mathscr{S}}_{t})_{t\in [0,T]}$ we denote the natural filtration of $\widetilde{\xi}$ and set $\mathscr{S}_{t} := S\cap \widetilde{\mathscr{S}}_{t}$ for all $t\in [0,T]$, which gives the natural filtration $(\mathscr{S}_{t})_{t\in [0,T]}$ of $\xi$. For each $t\in [0,T]$ and all $x\in \widetilde{S}$, we write
\[
\|x\|:=\sup_{s\in [0,T]} |x(s)|
\]
and let $x^{t}\in\widetilde{S}$ be the map $x$ stopped at time $t$. That is, $x^{t}(s) = x(s\wedge t)$ for all $s\in [0,T]$. Of course, $\widetilde{S}$ equipped with $\|\cdot\|$ is a Banach space, which, however, fails to be separable, and $S$ is a separable closed set in $\widetilde{S}$.

Due to the non-separability of $\widetilde{S}$ under the supremum norm and the fact that the Borel $\sigma$-field of $\widetilde{S}$ with respect to $\|\cdot\|$ is strictly larger than the cylindrical $\sigma$-field $\widetilde{\mathscr{S}}_{T}$, we equip $\widetilde{S}$ with a complete metric $\rho$ that induces the Skorohod topology and which satisfies $\rho(x,y) \leq \|x-y\|$ for all $x,y\in\widetilde{S}$. For instance, such a metric is introduced in Billingsley~\cite{BillingsleyOld}. Then $\widetilde{S}$ endowed with $\rho$ turns into a Polish space and the Borel $\sigma$-field of $\widetilde{S}$ with respect to $\rho$ is exactly $\widetilde{\mathscr{S}}_{T}$.

We recall that a map $u$ on $[0,T]\times\widetilde{S}$ is non-anticipative if $u(t,x) = u(t,x^{t})$ for all $t\in [0,T]$ and each $x\in\widetilde{S}$. Following Cont and Fourni\'{e}~\cite{ItoFormula} and using the setting in~\cite{PPDEViscosity}, we consider the complete pseudometric $d_{\infty}$ on $[0,T]\times\widetilde{S}$ given by
\[
d_{\infty}((r,x),(s,y)) := |r-s| + \|x^{r} - y^{s}\|.
\]
Then $d_{\infty}((r,x),(s,y)) = 0$ if and only if $r=s$ and $x^{r}=y^{r}$ for all $r,s\in [0,T]$ and each $x,y\in\widetilde{S}$. Thus, for $k\in\mathbb{N}$ an $\mathbb{R}^{k}$-valued map $u$ on $[0,T]\times\widetilde{S}$ that is continuous with respect to $d_{\infty}$ is automatically non-anticipative. However, there is no reason that $u$ is $(\widetilde{\mathscr{S}}_{t})_{t\in [0,T]}$-progressively measurable or at least $\mathscr{B}([0,T])\otimes\widetilde{\mathscr{S}}_{T}$-measurable. To circumvent this difficulty, we endow $[0,T]\times\widetilde{S}$ with the complete pseudometric $d_{S}$ defined via
\[
d_{S}((r,x),(s,y)):= |r -s| + \rho(x^{r},y^{s})
\]
under which $[0,T]\times\widetilde{S}$ becomes separable. Clearly, if $u$ is continuous\footnote{In comparison to~\cite{PPDEViscosity}, where $d_{\infty}$ is used, the choice of $d_{S}$ will have a negligible effect on the sizes of the test functions spaces we are going to use for our Definition~\ref{Viscosity Solutions} of viscosity solutions.} with respect to $d_{S}$, then continuity relative to $d_{\infty}$ follows and $u$ becomes $(\widetilde{\mathscr{S}}_{t})_{t\in [0,T]}$-progressively measurable, by Lemma~\ref{Measurability Lemma}. In fact, $u$ is already progressively measurable as soon as it is \emph{right-continuous} in the following sense: for each $(r,x)\in [0,T]\times\widetilde{S}$ and every $\varepsilon > 0$ there is $\delta > 0$ so that
\[
|u(s,y) - u(r,x)| < \varepsilon
\]
for all $(s,y)\in [r,T]\times \widetilde{S}$ with $d_{S}((s,y),(r,x)) < \delta$. If $d_{\infty}$ instead of $d_{S}$ is considered, then we will say that $u$ is right-continuous with respect to $d_{\infty}$. Moreover, the Borel $\sigma$-field of $[0,T]\times\widetilde{S}$ with respect to $d_{S}$ satisfies $\mathscr{B}([0,T]\times\widetilde{S})\subset\mathscr{B}([0,T])\otimes\widetilde{\mathscr{S}}_{T}$ and if $u$ is non-anticipative, then it is Borel measurable if and only if it is product measurable, as shown in Lemma~\ref{Measurability Lemma}.

Finally, let $B([0,T]\times\widetilde{S})$ be the linear space of all real-valued Borel measurable functions on $[0,T]\times\widetilde{S}$ and $B_{b}([0,T]\times\widetilde{S})$ be the set of all bounded $u\in B([0,T]\times\widetilde{S})$. By $C([0,T]\times\widetilde{S})$ we denote the set of all $u\in B([0,T]\times\widetilde{S})$ that are continuous (relative to $d_{S}$). With $C_{b}([0,T]\times\widetilde{S})$ the set of all bounded $u\in C([0,T]\times\widetilde{S})$ is meant. If $[0,T]$ is replaced by a non-degenerate interval $I$ in $[0,T]$ and $\widetilde{S}$ is replaced by a set $R\in\widetilde{\mathscr{S}}_{T}$ that is stable under stopping in the sense that $x^{t}\in R$ for all $(t,x)\in [0,T]\times R$, then non-anticipation, right-continuity, and the linear spaces $B(I\times R)$, $B_{b}(I\times R)$, $C(I\times R)$, and $C_{b}(I\times R)$ are defined in the same way.

\subsection{Differential calculus on path spaces}\label{path space calculus section}

Here we recall the definitions of the differential operators on path spaces that were introduced by Dupire~\cite{Dupire} and Cont and Fourni\'{e}~\cite{ItoFormula}. Again, we use the Cartesian setting in~\cite{PPDEViscosity}. For $k\in\mathbb{N}$ an $\mathbb{R}^{k}$-valued non-anticipative map $u$ on $[0,T)\times\widetilde{S}$ is called \emph{horizontally differentiable} if for each $(t,x)\in [0,T)\times\widetilde{S}$ the map
\[
[0,T-t)\rightarrow\mathbb{R}^{k}, \quad h\mapsto u(t + h,x^{t})
\]
is differentiable at $0$. Its derivative at $0$, known as the \emph{horizontal derivative} of $u$ at $(t,x)$, will be denoted by $\partial_{t}u(t,x)$. We say that $u$ is \emph{vertically differentiable} if for every $(t,x)\in [0,T)\times\widetilde{S}$ the map
\[
\mathbb{R}^{d}\rightarrow\mathbb{R}^{k}, \quad h\mapsto u(t,x + h\mathbbm{1}_{[t,T]})
\]
is differentiable at $0$. Its derivative at $0$ is called the \emph{vertical derivative} of $u$ at $(t,x)$ and will be represented by $\partial_{x} u(t,x)$. 

Let $\{e_{1},\dots,e_{d}\}$ be the standard basis of $\mathbb{R}^{d}$. Then $u$ is called \emph{partially vertically differentiable} if for all $i\in\{1,\dots,d\}$ and each $(t,x)\in [0,T)\times\widetilde{S}$, the map
\[
\mathbb{R}\rightarrow\mathbb{R}^{k},\quad h\mapsto u(t,x + he_{i}\mathbbm{1}_{[t,T]})
\]
is differentiable at $0$. Its derivative at $0$, which is called the \emph{$i$-th partial vertical derivative} of $u$ at $(t,x)$, will be written $\partial_{x_{i}} u(t,x)$. By calculus, if $u$ is vertically differentiable, then $u$ is partially vertically differentiable and $\partial_{x} u = (\partial_{x_{1}}u,\dots,\partial_{x_{d}} u)$.

For $k=1$ we say that the function $u$ is twice vertically differentiable if $u$ is vertically differentiable and the same is true for $\partial_{x}u$. In this case, we set
\[
\partial_{xx}u := \partial_{x}(\partial_{x}u)\quad\text{and}\quad \partial_{x_{i}x_{j}} u := \partial_{x_{i}} (\partial_{x_{j}} u) \quad \text{for all $i,j\in\{1,\dots,d\}$.}
\]
As a matter of fact, Schwarz's Lemma entails that whenever $u$ is twice vertically differentiable and $\partial_{xx}u$ is continuous with respect to $d_{\infty}$, then $\partial_{x_{j}x_{i}} u$ $ = \partial_{x_{i}x_{j}}u$ for all $i,j\in\{1,\dots,d\}$. Put differently, in this case $\partial_{xx}u$ is $\mathbb{S}^{d}$-valued, where $\mathbb{S}^{d}$ denotes the set of all real symmetric $d\times d$ matrices.

We define $C_{b}^{1,2}([0,T)\times\widetilde{S})$ to be the linear space of all $u\in C_{b}([0,T)\times\widetilde{S})$ that are once horizontally differentiable and twice vertically differentiable such that
\[
\partial_{t}u,\, \partial_{x_{i}}u,\, \partial_{x_{i}x_{j}} u \in C_{b}([0,T)\times\widetilde{S}) \quad \text{for all  $i,j\in\{1,\dots,d\}$.}
\]
Moreover, let $C_{b}^{1,2}([0,T)\times S)$ denote the linear space of all $u:[0,T)\times S\rightarrow\mathbb{R}$ for which there is $\widetilde{u}\in C_{b}^{1,2}([0,T)\times\widetilde{S})$ satisfying $u=\widetilde{u}$ on $[0,T)\times S$. The motivation of the latter space comes from the following fact. Let $u\in C_{b}^{1,2}([0,T)\times S)$ and suppose that $\widetilde{u}\in C_{b}^{1,2}([0,T)\times\widetilde{S})$ is an extension of $u$ to $[0,T)\times\widetilde{S}$. Then it follows from Theorem 2.3 in Fourni{\'e}~\cite{ItoCalculus} and the functional It{\^o} formula in~\cite{ItoFormula} that the definitions
\[
\partial_{t}u:= \partial_{t}\widetilde{u}, \quad \partial_{x}u := \partial_{x}\widetilde{u}, \quad \text{and} \quad \partial_{xx}u:= \partial_{xx}\widetilde{u} \quad \text{on } [0,T)\times S 
\]
are independent of the choice of $\widetilde{u}$. This has already been noted in~\cite{PPDEViscosity}[Theorem 2.4].

\subsection{The parabolic terminal value problem}\label{TVP Section}

In what follows, let $a:[0,T)\times S\rightarrow\mathbb{S}_{+}^{d}$ and $b:[0,T)\times S\rightarrow\mathbb{R}^{d}$ be two non-anticipative Borel measurable maps such that $a(\cdot,x)$ and $b(\cdot,x)$ are locally integrable for each $x\in S$. For example, this condition is satisfied if $a$ and $b$ are locally bounded. Here, $\mathbb{S}_{+}^{d}$ stands for the set of all positive definite matrices in $\mathbb{S}^{d}$. To $a$ and $b$ we associate the linear differential operator $\mathscr{L}:C_{b}^{1,2}([0,T)\times S)\rightarrow B([0,T)\times S)$ defined via
\begin{equation*}
\mathscr{L}(\varphi)(t,x) := \frac{1}{2}\sum_{i,j=1}^{d} a_{i,j}(t,x) \partial_{x_{i}x_{j}}\varphi(t,x) + \sum_{i=1}^{d}b_{i}(t,x) \partial_{x_{i}}\varphi(t,x).
\end{equation*}
Let $D\subset\mathbb{R}$ be a non-degenerate interval and $f:[0,T)\times S\times D\rightarrow\mathbb{R}$ be non-anticipative in the sense that $f(t,x,z)$ $= f(t,x^{t},z)$ for all $(t,x,z)\in [0,T)\times S\times D$. Moreover, let $f$ be $\mathscr{B}([0,T)\times S)\otimes\mathscr{B}(D)$-measurable and $g:S\rightarrow D$ be Borel measurable and bounded.

In this paper, we analyze the following semilinear parabolic path-dependent PDE combined with a terminal value condition:
\begin{equation}\label{Parabolic Terminal Value Problem}\tag{P}
\left\{
\begin{array}{rll}
(\partial_{t} + \mathscr{L})(u)(t,x) &= f(t,x,u(t,x)) &\text{for $(t,x)\in [0,T)\times S$,}\\
u(T,x) &= g(x) &\text{for $x\in S$.}
\end{array}
\right.
\end{equation}
By a \emph{classical subsolution} (resp.~\emph{supersolution}) to \eqref{Parabolic Terminal Value Problem} in $C_{b}^{1,2}([0,T)\times S)$ we mean a $D$-valued function $u\in C_{b}^{1,2}([0,T)\times S)\cap C([0,T]\times S)$ such that
\[
(\partial_{t} + \mathscr{L})(u)(t,x) \geq \text{(resp. $\leq$)}\,\, f(t,x,u(t,x)) \quad \text{and}\quad u(T,x) \leq \text{(resp. $\geq $)}\,\, g(x)
\]
for all $(t,x)\in [0,T)\times S$. Correspondingly, a \emph{classical solution} to \eqref{Parabolic Terminal Value Problem} in $C_{b}^{1,2}([0,T)\times S)$ is a $D$-valued $u\in C_{b}^{1,2}([0,T)\times S)\cap C([0,T]\times S)$ that is a classical sub- and supersolution to \eqref{Parabolic Terminal Value Problem} in the same space. 

However, classical solutions may not exist in many applications. This can already be seen from \eqref{mild solution integral eq intro} in the linear case $f=0$, where martingale arguments yield the representation 
$u(r,x)$ $=E_{r,x}[g(X^T)]$ for all $(r,x)\in [0,T]\times S$ whenever $u$ is a classical solution in $C_{b}^{1,2}([0,T)\times S)$. Since $X^T$ coincides $P_{r,x}$-a.s.~with the function $x$ up to time $r$, one sees that even continuity of the terminal condition $g$ may not suffice if $u$ shall belong to $C_{b}^{1,2}([0,T)\times S)$. For example, let $g$ be  of the form $g(x) = \overline{g}(x(t))$ for all $x\in S$, some $D$-valued $\overline{g}\in C_{b}(\mathbb{R}^{d})$ with $\overline{g}\neq 0$, and some $t\in [0,T)$, then $u\in C_{b}^{1,2}([0,T)\times S)$ implies that $\overline{g}\in C_{b}^{2}(\mathbb{R}^{d})$.

But such a requirement is too strong for many applications. For instance, in mathematical finance, $g$ could correspond to the payoff of a path-dependent derivative, and these payoffs are often not smooth but merely continuous functions of the underlying path. For this reason,  it is natural to focus on \emph{weak} solutions to \eqref{Parabolic Terminal Value Problem}. Here, our main interest is in  mild and viscosity solutions, which will be introduced in the next two sections.  Existence and uniqueness results for classical solutions were given by Peng and Wang \cite{Peng2011} and Ji and Yang \cite{Ji2013}.

\subsection{Diffusion processes, mild solutions, and a control problem}\label{mild solution section}

In what follows, we require the notion of an $\mathscr{L}$-diffusion process. At first, a \emph{path-dependent diffusion process} on some measurable space $(\Omega,\mathscr{F})$ is a triple $\mathscr{X}=(X,(\mathscr{F}_{t})_{t\in [0,T]},\mathbb{P})$ that is composed of a continuous process $X:[0,T]\times\Omega\rightarrow\mathbb{R}^{d}$, a filtration $(\mathscr{F}_{t})_{t\in [0,T]}$ of $\mathscr{F}$ to which $X$ is adapted, and a set $\mathbb{P}=\{P_{r,x}\,|\, (r,x)\in [0,T]\times S\}$ of probability measures on $(\Omega,\mathscr{F})$ such that for the path process $\hat{X}$ of $X$ given by $\hat{X}_{t}:= X^{t}$ for all $t\in [0,T]$ the triple
\[
\hat{\mathscr{X}}:=(\hat{X},(\mathscr{F}_{t})_{t\in [0,T]},\mathbb{P})
\]
is a non-anticipative diffusion process on $(\Omega,\mathscr{F})$ with state space $S$. As $\hat{X}$ is automatically continuous, this results in the additional requirement that $\hat{\mathscr{X}}$ is a non-anticipative strong Markov process. That means, the subsequent three conditions hold:
\begin{enumerate}[(i)]
\item $P_{r,x} = P_{r,x^{r}}$ and $\hat{X}_{r} = x^{r}$ $P_{r,x}$-a.s. for each $(r,x)\in [0,T]\times S$.
\item The function $[0,t]\times S\rightarrow [0,1]$, $(s,y)\mapsto P_{s,y}(\hat{X}_{t}\in B)$ is Borel measurable for all $t\in [0,T]$ and each $B\in\mathscr{S}_{T}$.
\item $P_{r,x}(\hat{X}_{t}\in B|\mathscr{F}_{\tau}) = P_{\tau,\hat{X}_{\tau}}(\hat{X}_{t}\in B)$ $P_{r,x}$-a.s.~for all $r,t\in [0,T]$ with $r\leq t$, each finite $(\mathscr{F}_{s})_{s\in [r,t]}$-stopping time $\tau$, every $x\in S$, and all $B\in\mathscr{S}_{T}$.
\end{enumerate}
This notion includes in particular the class of path or historical processes used by Dawson and Perkins~\cite{HistProcessesPerkins} and Dynkin~\cite{PathProcesses} for constructing historical superprocesses; see Example~\ref{Ordinary Diffusion Processes} below for details. Furthermore, an $\mathscr{L}$\emph{-diffusion process} is a path-dependent diffusion process $\mathscr{X}$ with the $\mathscr{L}$-martingale property:
\begin{enumerate}[$(\mathscr{L})$]
\item The process $[r,T)\times\Omega\rightarrow\mathbb{R}$, $(t,\omega)\mapsto \varphi(t,X^{t}(\omega)) - \int_{r}^{t}(\partial_{s} + \mathscr{L})(\varphi)(s,X^{s}(\omega))\,ds$ is an $(\mathscr{F}_{t})_{t\in [r,T)}$-martingale under $P_{r,x}$ for each $(r,x)\in [0,T)\times S$ and every $\varphi\in C_{b}^{1,2}([0,T)\times S)$.
\end{enumerate}

We suppose that $\mathscr{X}$ is an $\mathscr{L}$-diffusion process on some measurable space $(\Omega,\mathscr{F})$ and $u$ is a classical subsolution (resp.~supersolution) to \eqref{Parabolic Terminal Value Problem} in $C_{b}^{1,2}([0,T)\times S)$. Then we can accomplish that
\begin{align*}
E_{r,x}[u(t\wedge \tau,X^{t\wedge\tau})]  - u(r,x) &= E_{r,x}\bigg[\int_{r}^{t\wedge\tau}(\partial_{s} + \mathscr{L})(u)(s,X^{s})\,ds\bigg]\\
&\geq \text{(resp. $\leq$)}\,\, E_{r,x}\bigg[\int_{r}^{t\wedge\tau}f(s,X^{s},u(s,X^{s}))\,ds\bigg]
\end{align*}
for all $r,t\in [0,T)$ with $r\leq t$, every $(\mathscr{F}_{s})_{s\in [r,T]}$-stopping time $\tau$, and each $x\in S$, due to optional sampling. Hence, if $\tau$ is finite and $\int_{r}^{\tau}|f(s,X^{s},u(s,X^{s}))|\,ds$ is a finite $P_{r,x}$-integrable function, then we may take the limit $t\uparrow T$ to obtain that
\[
E_{r,x}[u(\tau,X^{\tau})] - u(r,x) \geq \text{(resp. $\leq$)}\,\, E_{r,x}\bigg[\int_{r}^{\tau}f(s,X^{s},u(s,X^{s}))\,ds\bigg],
\]
by dominated convergence. This motivates notions of mild sub- and supersolutions as well as mild solutions to \eqref{Parabolic Terminal Value Problem}.

\begin{Definition}
A \emph{mild subsolution} (resp.~\emph{supersolution}) to \eqref{Parabolic Terminal Value Problem} is a $D$-valued non-anticipative function $u\in B([0,T]\times S)$ for which $|u(\tau,X^{\tau})| + \int_{r}^{\tau}|f(s,X^{s},u(s,X^{s}))|\,ds$ is finite and $P_{r,x}$-integrable such that
\[
E_{r,x}[u(\tau,X^{\tau})] - u(r,x) \geq \text{(resp. $\leq$)}\,\, E_{r,x}\bigg[\int_{r}^{\tau}f(s,X^{s},u(s,X^{s}))\,ds\bigg]
\]
for all $(r,x)\in [0,T]\times S$ and each finite $(\mathscr{F}_{t})_{t\in [r,T]}$-stopping time $\tau$. In addition, we require that $u(T,x) \leq \text{(resp.~$\geq $) } g(x)$ for every $x\in S$. A \emph{mild solution} to \eqref{Parabolic Terminal Value Problem} is a $D$-valued $u\in B([0,T]\times S)$ that is a mild sub- and supersolution to \eqref{Parabolic Terminal Value Problem}.
\end{Definition}

The strong Markov property of $\hat{\mathscr{X}}$ gives a characterization of mild solutions.

\begin{Lemma}\label{Characterization of Mild Solutions}
A $D$-valued function $u\in B([0,T]\times S)$ is a mild solution to \eqref{Parabolic Terminal Value Problem} if and only if $\int_{r}^{T}|f(s,X^{s},u(s,X^{s}))|\,ds$ is a finite $P_{r,x}$-integrable function such that
\begin{equation}\label{mild sol int eq}
E_{r,x}[g(X^{T})] = u(r,x) + E_{r,x}\bigg[\int_{r}^{T}f(s,X^{s},u(s,X^{s}))\,ds\bigg]
\end{equation}
for all $(r,x)\in [0,T]\times S$.
\end{Lemma}

Since we deal with a non-degenerate interval $D$, we cite the following existence result from \cite{Kalinin}[Theorem 2.11] that is a direct consequence of the Markov property of $\hat{\mathscr{X}}$ and the definition of a global solution to a Markovian integral equation. Here, we say that $\hat{\mathscr{X}}$ is (right-hand) Feller if the function 
\[
[0,t]\times S\rightarrow\mathbb{R},\quad (r,x)\mapsto E_{r,x}[\varphi(X^{t})]
\]
is (right-)continuous for all $t\in [0,T]$ and each $\varphi\in C_{b}(S)$. Moreover, we call $f$ right-continuous in the same sense as before. That is, for each $(r,x,z)\in [0,T)\times S\times D$ and every $\varepsilon > 0$, there is $\delta > 0$ such that $|f(s,y,z') - f(r,x,z)| < \varepsilon$ for all $(s,y,z')\in [r,T)\times S\times D$ with $d_{S}((s,y),(r,x)) + |z'-z| < \delta$. If $f$ is merely right-continuous at each $(r,x,z)\in N^{c}\times S\times D$, where $N\in\mathscr{B}([0,T))$ is a Lebesgue-null set, then we will say that $f$ is a.s.~right-continuous.

\begin{Theorem}\label{mild sol existence thm}
Let $\mathscr{X}$ be an $\mathscr{L}$-diffusion process, $\underline{d}:=\inf D$, and $\overline{d}:=\sup D$. Suppose that $f$ satisfies the following three conditions:
\begin{enumerate}[(i)]
\item For each $\hat{z}\in\overline{D}$ there are $\delta > 0$ and two integrable $\kappa,\lambda:[0,T)\rightarrow\mathbb{R}_{+}$ such that $|f(\cdot,x,z)|$ $\leq \kappa$ and $|f(\cdot,x,z) - f(\cdot,x,z')| \leq \lambda|z-z'|$ for all $x\in S$ and each $z,z'\in (\hat{z}-\delta,\hat{z}+\delta)\cap D$ a.s.~on $[0,T)$.
\item If $\underline{d}=-\infty$ (resp.~$\overline{d}=\infty$), then there are two integrable $\alpha,\beta:[0,T)\rightarrow\mathbb{R}_{+}$ with $f(\cdot,x,z)$ $\leq \alpha + \beta|z|$ (resp.~$f(\cdot,x,z) \geq - \alpha - \beta|z|$) for all $(x,z)\in S\times D$ a.s.~on $[0,T)$.
\item Whenever $\underline{d}> -\infty$ (resp. $\overline{d} < \infty$), then $\lim_{z\downarrow\underline{d}} f(\cdot,x,z)\leq 0$ (resp. $\lim_{z\uparrow\overline{d}} f(\cdot,x,z)\geq 0$) for all $x\in S$ a.s.~on $[0,T)$. 
\end{enumerate}
Then there is a unique bounded mild solution $u$ to \eqref{Parabolic Terminal Value Problem}. Moreover, if $\hat{\mathscr{X}}$ is (right-hand) Feller, $f$ is a.s.~right-continuous, and $g\in C_{b}(S)$, then $u$ is (right-)continuous.
\end{Theorem}

For instance, the following case is included: $D=\mathbb{R}_{+}$ and $f(t,x,z) = \alpha(t,x)f(z)$ for every $(t,x,z)\in [0,T)\times S\times\mathbb{R}_{+}$ with non-negative and non-anticipative $\alpha\in B_{b}([0,T)\times S)$ as well as locally Lipschitz continuous $f:\mathbb{R}_{+}\rightarrow\mathbb{R}$ that is bounded from below and satisfies $f(0)\leq 0$. In particular, one can take $f(z) = z^{p}$ for all $z\in\mathbb{R}_{+}$ with $p \geq 1$ (see \cite{Kalinin} for further examples).

% Pazy~\cite[Theorem 6.1.4]{Pazy} is missing.

\begin{Remark}{\bf (Existence of mild solutions)}\label{mild sol existence rem}
Let $D=\mathbb{R}_{+}$. Based on \eqref{mild sol int eq}, the existence of mild solutions to \eqref{Parabolic Terminal Value Problem} can be proved by standard Picard iteration; see, e.g., Watanabe~\cite[Proposition 2.2]{Watanabe}, Fitzsimmons~\cite[Proposition 2.3]{Fitzsimmons}, Iscoe~\cite[Theorem A]{Iscoe} for corresponding results with various degrees of generality. 

While in \cite{Watanabe,Fitzsimmons,Iscoe} the existence of mild solutions is used to construct superprocesses, Dynkin~\cite{BranchingParticle,PathProcesses,DynkinBranchingBook} derives mild solutions to \eqref{Parabolic Terminal Value Problem} from a probabilistic construction of superprocesses. This works for  
nonlinearities of the form
\begin{equation}\label{Superprocess Function}
f(t,x,z) = \alpha(t,x)z +  \gamma(t,x)z^{2} + \int_{0}^{\infty} (e^{-uz} - 1 + uz)\,n(t,x,du)
\end{equation}
for each $(t,x,z)\in [0,T)\times S\times\mathbb{R}_{+}$. Here, $\alpha,\gamma\in B_{b}([0,T)\times S)$ are non-negative and non-anticipative, and $n$ is a Markov kernel from $[0,T)\times S$ to $(0,\infty)$ for which the function $[0,T)\times S\rightarrow [0,\infty]$, $(t,x)\mapsto \int_{(0,\infty)} u\wedge u^{2}\, n(t,x,du)$ is bounded. Note that  \eqref{Superprocess Function} includes as special case functions of the form 
\[
f(t,x,z) = \alpha(t,x)z+ \sum_{i=1}^{n}  \beta_{i}(t,x)z^{v_{i}} +  \gamma(t,x)z^{2}
\]
for all $(t,x,z)\in [0,T)\times S\times\mathbb{R}_{+}$, where $n\in\mathbb{N}$, $v_{1},\dots,v_{n}\in (1,2)$, and $\beta_{1},\dots,\beta_{n}\in B_b([0,T)\times S)$ are non-negative and non-anticipative.

Finally, we mention that for $D=\mathbb{R}$ mild solutions can also be constructed by means of backward stochastic differential equations \cite{PPDEViscosity,Peng2011,Ji2013}, but this method typically requires global Lipschitz continuity of the function $\mathbb{R}\rightarrow\mathbb{R}$, $z\mapsto f(t,x,z)$, which in particular excludes the case $f(z)=|z|^p$ for all $z\in\mathbb{R}$ with $p>1$.
\end{Remark}

It was observed in \cite{SchiedFuel,NeumanSchied} that the value functions and the optimal strategies in certain classes of stochastic optimal control problems can be represented in terms of mild solutions to semilinear terminal value problems. We illustrate this idea with the following proposition, which partly extends \cite[Theorem 2.8]{SchiedFuel} to all powers $p>1$ (note, however, that we assume bounds on the coefficients of the cost functional so as to keep the exposition simple; the reader will have no difficulties in relaxing these bounds if needed). We refer to \cite{Almgren,Forsythetal,NeumanSchied} for a financial motivation of the control problem and to \cite{AnkirchnerJeanblancKruse,GraeweHorstQiu} for   solution methods by means of BSDEs. Since  value functions of control problems are typically viscosity solutions of the corresponding HJB equations, the following proposition also motivates our main research question, namely whether a mild solution to \eqref{Parabolic Terminal Value Problem} is also a viscosity solution.

\begin{Proposition}\label{control prop}Suppose that $p>1$, $q$ is its dual exponent defined by $\frac1p+\frac1q=1$, $D=\mathbb{R}_{+}$, and $\alpha,\eta\in B_{b}([0,T)\times S)$ are non-negative, where $\eta$ is bounded away from zero. Let $f$ be of the form
\[
f(t,x,z)=-\alpha(t,x)+\frac{z^q}{(q-1)\eta(t,x)^{q-1}}
\]
for all $(t,x,z)\in [0,T)\times S\times\mathbb{R}_{+}$ and $u$ be the unique mild solution to \eqref{Parabolic Terminal Value Problem}. If $u$ is right-continuous on $[0,T)\times S$, then
\[
\nu^{*}(t)=\nu_{0}\exp\bigg(-\int_{0}^{t}\Big(\frac{u(s,X^{s})}{\eta(s,X^{s})}\Big)^{q-1}\,ds\bigg)
\]
is the $P_{0,x}$-a.s.~unique minimizer of the cost functional 
\[
J(\nu)=E_{0,x}\bigg[\int_{0}^{T}\Big(|\dot\nu(t)|^{p}\eta(t,X^{t})+|\nu(t)|^{p}\alpha(t,X^{t})\Big)\,dt+g(X^{T})|\nu(T)|^{p}\bigg]
\]
within the class of control processes $(\nu(t))_{t\in [0,T]}$ that are of the form $\nu(t)=\nu_0+\int_0^t\dot\nu(s)\,ds$ for some given constant $\nu_0\in\mathbb{R}$ and a progressive and integrable process $(\dot\nu(t))_{t\in [0,T]}$.
\end{Proposition}

We conclude this section with the following example, which explains how the path process of a standard Markovian diffusion process fits into our framework of $\mathscr{L}$-diffusion processes.

\begin{Example}\label{Ordinary Diffusion Processes}
Assume that there are two Borel measurable maps $\overline{a}:[0,T)\times\mathbb{R}^{d}\rightarrow\mathbb{S}_{+}^{d}$ and $\overline{b}:[0,T)\times\mathbb{R}^{d}\rightarrow\mathbb{R}^{d}$ such that $a(t,x) = \overline{a}(t,x(t))$ and $b(t,x) = \overline{b}(t,x(t))$ for all $(t,x)\in [0,T)\times S$. With $\overline{a}$ and $\overline{b}$ we can link the linear differential operator $\overline{\mathscr{L}}:C^{0,2}([0,T)\times\mathbb{R}^{d})\rightarrow B([0,T)\times\mathbb{R}^{d})$ given by
\[
\overline{\mathscr{L}}(\varphi)(t,\overline{x}) := \frac{1}{2}\sum_{i,j=1}^{d}\overline{a}_{i,j}(t,\overline{x})\frac{\partial^{2}\varphi}{\partial x_{i}\partial x_{j}}(t,\overline{x}) + \sum_{i=1}^{d} \overline{b}_{i}(t,\overline{x})\frac{\partial\varphi}{\partial x_{i}}(t,\overline{x}).
\]
Suppose that there is a set of probability measures $\overline{\mathbb{P}}=\{\overline{P}_{r,\overline{x}}\,|\,(r,\overline{x})\in [0,T]\times\mathbb{R}^{d}\}$ for which $(\xi,(\mathscr{S}_{t})_{t\in [0,T]},\overline{\mathbb{P}})$ becomes a canonical $\overline{\mathscr{L}}$-diffusion process in the standard sense. In other words, $(\xi,(\mathscr{S}_{t})_{t\in [0,T]},\overline{\mathbb{P}})$ is a diffusion process on $(S,\mathscr{S}_{T})$ with state space $\mathbb{R}^{d}$ for which the $\overline{\mathscr{L}}$-martingale property holds, that is, the process $[r,T)\times S\rightarrow\mathbb{R}$,
\[
(t,x)\mapsto \varphi(t,\xi_{t}(x)) - \int_{r}^{t}\bigg(\frac{\partial}{\partial s} + \overline{\mathscr{L}}\bigg)(\varphi)(s,\xi_{s}(x))\,ds
\]
is an $(\mathscr{S}_{t})_{t\in [r,T)}$-martingale under $\overline{P}_{r,\overline{x}}$ for each $(r,\overline{x})\in [0,T)\times\mathbb{R}^{d}$ and every $\varphi\in C_{b}^{1,2}([0,T)\times\mathbb{R}^{d})$. Then for each $(r,x)\in [0,T]\times S$ we let $P_{r,x}$ denote the unique probability measure on $(S,\mathscr{S}_{T})$ with $\xi^{r} = x^{r}$ $P_{r,x}$-a.s. such that the law of $\xi$ restricted to $[r,T]\times S$ under $\overline{P}_{r,x(r)}$ remains the same under $P_{r,x}$.

By setting $\mathbb{P}:=\{P_{r,x}\,|\, (r,x)\in [0,T]\times S\}$ and recalling that $\hat{\xi}_{t}=\xi^{t}$ for all $t\in [0,T]$, it follows that $(\hat{\xi},(\mathscr{S}_{t})_{t\in [0,T]},\mathbb{P})$ is a non-anticipative diffusion process with state space $S$. This procedure appears for instance in the construction of historical superprocesses (see~\cite{HistProcessesPerkins},~\cite{BranchingParticle}, and~\cite{PathProcesses}). Moreover, it follows from the functional It\^{o} formula that $(\xi,(\mathscr{S}_{t})_{t\in [0,T]},\mathbb{P})$ is an $\mathscr{L}$-diffusion process on $(S,\mathscr{S}_{T})$ as specified in the beginning of this section. 
\end{Example}

\subsection{Test functions for viscosity solutions}\label{Test functions for viscosity solutions}

For the introduction of several test function spaces below, we let $\mathscr{T}$ denote the set of all finite $(\mathscr{F}_{t})_{t\in [0,T]}$-stopping times $\tau$ for which there is a lower semicontinuous function $\phi:S\rightarrow [0,T]$ such that $\tau(\omega) = \phi(X(\omega))$ for all $\omega\in\Omega$. Put differently, a finite $(\mathscr{F}_{t})_{t\in [0,T]}$-stopping time is a member of $\mathscr{T}$ if and only if for each $t\in [0,T)$ there is an open set $O_{t}$ in $S$ such that
\[
\{\tau > t\} = \{\omega\in\Omega\,|\,X(\omega)\in O_{t}\}.
\]
If $\mathscr{X}$ is canonical, that is, $(\Omega,\mathscr{F}) = (S,\mathscr{S}_{T})$, $X=\xi$, and $\mathscr{F}_{t} = \mathscr{S}_{t}$ for all $t\in [0,T]$, then our definition of $\mathscr{T}$ reduces to that in~\cite{PPDEViscosity}. For example, $\tau:=\inf\{s\in [r,T]\,|\, u(s,X^{s})\in I\}\wedge t$ belongs to $\mathscr{T}$ for all $r,t\in [0,T]$ with $r\leq t$, each $u\in C([r,T]\times S)$, and every closed interval $I$ in $\mathbb{R}$.

For each non-anticipative function $u\in B_{b}([0,T]\times S)$ and every $(r,x)\in [0,T)\times S$, we define $\underline{\mathscr{S\!P}}\,u(r,x)$ to be the set of all $\varphi\in C_{b}^{1,2}([0,T)\times S)$ for which there is an $(\mathscr{F}_{t})_{t\in [r,T]}$-stopping time $\tau$ with $P_{r,x}(\tau > r) > 0$ such that
\[
(u-\varphi)(r,x) \geq E_{r,x}[(u-\varphi)(\widetilde{\tau}\wedge\tau,X^{\widetilde{\tau}\wedge\tau})]
\]
for every $\widetilde{\tau}\in\mathscr{T}$ with $\widetilde{\tau}\in [r,r+\delta)$ and some $\delta \in (0,T-r)$. In addition, we set $\overline{\mathscr{S\!P}}\, u(r,x):= -\underline{\mathscr{S\!P}}\,(-u)(r,x)$. Let $\underline{\mathscr{P}}\,u(r,x)$ be the set of all $\varphi\in C_{b}^{1,2}([0,T)\times S)$ such that $u-\varphi$ has a right-hand local maximum at $(r,x)$ in the sense that
\[
(u-\varphi)(r,x) \geq (u-\varphi)(s,y)
\]
for all $(s,y)\in [r,T)\times S$ with $d_{S}((s,y),(r,x)) < \delta$ and some $\delta\in (0,T-r)$. Moreover, we set $\overline{\mathscr{P}}\,u(r,x) := -\underline{\mathscr{P}}(-u)(r,x)$.

\begin{Definition}\label{Viscosity Solutions}
Let $u\in B_{b}([0,T]\times S)$ be $D$-valued and non-anticipative.
\begin{enumerate}[(i)]
\item We call $u$ a \emph{stochastic viscosity subsolution} (resp.~\emph{supersolution}) to \eqref{Parabolic Terminal Value Problem} if for every $(r,x)\in [0,T)\times S$ and each $\varphi\in\underline{\mathscr{S\!P}}\,u(r,x)$ (resp.~$\varphi\in\overline{\mathscr{S\!P}}\,u(r,x)$),
\[
(\partial_{r} + \mathscr{L})(\varphi)(r,x) \geq \text{(resp.~$\leq$) } f(r,x,u(r,x))\quad \text{and}\quad u(T,x) \leq \text{(resp.~$\geq $) } g(x).
\]
Moreover, $u$ is said to be a \emph{stochastic viscosity solution} to \eqref{Parabolic Terminal Value Problem} if it is both a stochastic viscosity sub- and supersolution.
\item We say that $u$ is a \emph{right-hand viscosity subsolution} (resp.~\emph{supersolution}) to \eqref{Parabolic Terminal Value Problem} if for all $(r,x)\in [0,T)\times S$ and each $\varphi\in\underline{\mathscr{P}}\,u(r,x)$ (resp.~$\varphi\in\overline{\mathscr{P}}\,u(r,x)$),
\[
(\partial_{r} + \mathscr{L})(\varphi)(r,x) \geq \text{(resp.~$\leq$) } f(r,x,u(r,x))\quad \text{and}\quad u(T,x) \leq \text{(resp.~$\geq $) } g(x).
\]
Furthermore, $u$ is a \emph{right-hand viscosity solution} to \eqref{Parabolic Terminal Value Problem} if it is a right-hand viscosity sub- and supersolution.
\end{enumerate}
\end{Definition}

By definition, every right-hand viscosity subsolution (resp.~supersolution) that belongs to $C_{b}^{1,2}([0,T)\times S)\cap C([0,T]\times S)$ is a classical subsolution (resp.~supersolution) in $C_{b}^{1,2}([0,T)\times S)$. As will be shown in Section~\ref{Relations between the notions of viscosity solutions}, every stochastic viscosity subsolution (resp.~supersolution) is a right-hand viscosity subsolution (resp.~supersolution). There we will also discuss the relations between the notion of a viscosity solution in~\cite{PPDEViscosity} and Definition~\ref{Viscosity Solutions}.

\subsection{The main results}\label{Main results Section}

We first state a general result on the relation between mild and viscosity solutions.

\begin{Theorem}\label{Main Result 1}
Let $\mathscr{X}$ be an $\mathscr{L}$-diffusion process and $a$, $b$, and $f$ be right-continuous. Then every bounded mild subsolution (resp.~supersolution) to \eqref{Parabolic Terminal Value Problem} that is right-continuous on $[0,T)\times S$ is a stochastic viscosity subsolution (resp.~supersolution).
\end{Theorem}

After the authors finished a previous version of this paper, they became aware of \cite{CossoEtal2018}, where results were obtained that partially precede our Theorem \ref{Main Result 1}.  
In combination with Theorem~\ref{mild sol existence thm}, this immediately yields an existence result for viscosity solutions.

\begin{Corollary}
Let $\mathscr{X}$ be an $\mathscr{L}$-diffusion process for which $\hat{\mathscr{X}}$ is (right-hand) Feller, and let $a$, $b$, and $f$ be right-continuous. Assume that $f$ satisfies the hypotheses of Theorem~\ref{mild sol existence thm} and $g\in C_{b}(S)$. Then there exists a bounded (right-)continuous stochastic viscosity solution to \eqref{Parabolic Terminal Value Problem}.
\end{Corollary}

\begin{Remark}
In \cite{KalininDiff} an $\mathscr{L}$-diffusion process $\mathscr{X}$ for which $\hat{\mathscr{X}}$ is Feller is derived under the following condition: there are three square-integrable $\alpha,\beta,\lambda:[0,T)\rightarrow\mathbb{R}_{+}$ such that
\begin{align*}
|b(\cdot,x)|\vee|\sigma(\cdot,x)| \leq \alpha + \beta\|x\|\quad\text{and}\quad |b(\cdot,x) - b(\cdot,y)|\vee |\sigma(\cdot,x) - \sigma(\cdot,y)| \leq \lambda\|x-y\|
\end{align*}
for each $x,y\in S$ a.s.~on $[0,T)$. Here, $\sigma:[0,T)\times S\rightarrow\mathbb{R}^{d\times d}$ is a non-anticipative Borel measurable map that satisfies $a = \sigma\sigma^{t}$, and $|\cdot|$ also stands for the Frobenius norm on $\mathbb{R}^{d\times d}$.
\end{Remark}

In the case that the underlying PPDE in \eqref{Parabolic Terminal Value Problem} is affine, we obtain stronger results than in the general case but at the cost of a more technical proof. Here, we use the right-hand upper and lower semicontinuous envelopes of a right-hand locally bounded function $u\in B([0,T]\times S)$ that are respectively given by
\[
u^{\leftarrow}(r,x) = \limsup_{(s,y)\rightarrow (r,x),\, s\geq r} u(s,y)\quad\text{and}\quad u_{\leftarrow}(r,x) = \liminf_{(s,y)\rightarrow (r,x),\, s\geq r} u(s,y)
\]
for all $(r,x)\in [0,T]\times S$.

\begin{Theorem}\label{Main Result 2}
Assume that $\mathscr{X}$ is an $\mathscr{L}$-diffusion process and $a$ and $b$ are right-continuous. Let $f(t,x,z) = \alpha(t,x) + \beta(t,x)z$ for all $(t,x,z)\in [0,T)\times S\times D$ and some right-continuous $\alpha,\beta:[0,T)\times S\rightarrow\mathbb{R}$ for which $\alpha(\cdot,x)$ and $\beta(\cdot,x)$ are locally integrable for all $x\in S$. Then the following two assertions hold for each bounded mild solution $u$ to \eqref{Parabolic Terminal Value Problem}:
\begin{enumerate}[(i)]
\item $u$ is a stochastic viscosity solution regardless of whether it is right-continuous on $[0,T)\times S$.
\item Let $D$ be closed and suppose that $\lim_{n\uparrow\infty} P_{r_{n},x_{n}}(\|X^{t_{n}} - x_{n}^{r_{n}}\|\geq\gamma) = 0$ for all $\gamma > 0$, every $(r,x)\in [0,T)\times S$, and each sequence $(r_{n},x_{n},t_{n})_{n\in\mathbb{N}}$ in $[r,T)\times S\times [r,T)$ with $t_{n}\geq r_{n}$ for all $n\in\mathbb{N}$ and $\lim_{n\uparrow\infty} (r_{n},x_{n},t_{n}) = (r,x,r)$. Then $u^{\leftarrow}$ (resp.~$u_{\leftarrow}$) is a right-hand viscosity subsolution (resp.~supersolution).
\end{enumerate}
\end{Theorem}

\begin{Remark}\label{Superprocess example}
In the affine case, there is the following Feynman--Kac formula for the bounded mild solution $u$. To state it, suppose in the context of Theorem~\ref{Main Result 2}  that  $|\alpha(\cdot,x)|\vee|\beta(\cdot,x)| \leq \gamma$ for all $x\in S$ a.s.~on $[0,T)$ and some integrable function $\gamma:[0,T)\rightarrow\mathbb{R}_{+}$. Then
\[
u(r,x) = E_{r,x}\bigg[e^{-\int_{r}^{T}\beta(s,X^{s})\,ds}g(X^{T})\bigg] - E_{r,x}\bigg[\int_{r}^{T}e^{-\int_{r}^{t}\beta(s,X^{s})\,ds}\alpha(t,X^{t})\,dt\bigg]
\]
for all $(r,x)\in [0,T]\times S$. For $\alpha=0$ and $\beta\geq 0$ this result follows from \cite[Theorem 4.1.1]{DynkinBranchingBook}; the general case is proved in \cite{Kalinin}[Corollary 2.13].
\end{Remark}

\subsection{Relations between the notions of viscosity solutions}\label{Relations between the notions of viscosity solutions}

To discuss the announced relations between the notion of a viscosity solution in~\cite{PPDEViscosity} and our Definition~\ref{Viscosity Solutions}, we consider the following. For each $(r,x)\in [0,T)\times S$ and every $(\mathscr{F}_{t})_{t\in [r,T]}$-progressively measurable process $\beta:[r,T]\times\Omega\rightarrow\mathbb{R}^{d}$ fulfilling
\begin{equation}\label{Local Martingale Condition}
\int_{r}^{T}\bigg|\sum_{i=1}^{d}\beta_{s}^{(i)} b_{i}(s,X^{s})\bigg|\,ds + \int_{r}^{T} \sum_{i,j=1}^{d} a_{i,j}(s,X^{s})\beta_{s}^{(i)}\beta_{s}^{(j)}\,ds < \infty,
\end{equation}
we pick an $(\mathscr{F}_{t})_{t\in [r,T]}$-progressively measurable process $M^{r,\beta}:[r,T]\times\Omega\rightarrow (0,\infty)$ with right-continuous and $P_{r,x}$-a.s.~continuous paths so that
\[
M_{t}^{r,\beta} = \exp\bigg(\int_{r}^{t}\beta_{s}\,dX_{s} - \int_{r}^{t}\sum_{i=1}^{d}\beta_{s}^{(i)}b_{i}(s,X^{s})\,ds -  \frac{1}{2}\int_{r}^{t}\sum_{i,j=1}^{d} a_{i,j}(s,X^{s})\beta_{s}^{(i)}\beta_{s}^{(j)}\,ds\bigg)
\]
for all $t\in [r,T]$ $P_{r,x}$-a.s. Here, the stochastic integral is constructed using Lemma 4.3.3 and Exercise 4.6.8 of~\cite{DiffusionProcesses}, which ensures right-continuity of all paths such that the set of all continuous paths has $P_{r,x}$-measure one; these considerations are needed, as we do not impose the usual conditions. Clearly, $M^{r,\beta}$ is an $(\mathscr{F}_{t})_{t\in [r,T]}$-supermartingale under $P_{r,x}$ that becomes an $(\mathscr{F}_{t})_{t\in [r,T]}$-martingale if and only if $E_{r,x}\big[M_{T}^{r,\beta}\big] = 1$. For $L \geq 0$ we let $\mathscr{U}_{r}^{L}$ be the set of all $(\mathscr{F}_{t})_{t\in [r,T]}$-progressively measurable processes $\beta:[r,T]\times\Omega\rightarrow\mathbb{R}^{d}$ for which each coordinate function is bounded by $L$ and $E_{r,x}\big[M_{T}^{r,\beta}\big] = 1$.

For each non-anticipative function $u\in B_{b}([0,T]\times S)$ let $\underline{\mathscr{A}}^{L} u(r,x)$ be the set of all $\varphi\in C_{b}^{1,2}([0,T)\times S)$ for which there is $\tau\in\mathscr{T}$ with $\tau > r$ $P_{r,x}$-a.s. such that
\[
(u-\varphi)(r,x) \geq E_{r,x}[M_{T}^{r,\beta}(u-\varphi)(\widetilde{\tau}\wedge\tau,X^{\widetilde{\tau}\wedge\tau})] 
\]
for all $\widetilde{\tau}\in\mathscr{T}$ with $\widetilde{\tau}\in [r,r+\delta)$, every $\beta\in\mathscr{U}_{r}^{L}$, and some $\delta\in (0,T-r)$. Correspondingly, we set $\overline{\mathscr{A}}^{L}u(r,x) := -\underline{\mathscr{A}}^{L}(-u)(r,x)$. This translates the concepts and spaces of test functions used for the definition of a viscosity solution in~\cite{PPDEViscosity} to our current setting. Hence, $u$ is a viscosity subsolution (resp.~supersolution) to \eqref{Parabolic Terminal Value Problem} in the sense\footnote{Note that only continuous functions $u$ are considered in~\cite{PPDEViscosity}.} of~\cite{PPDEViscosity} if there is $L\geq 0$ such that
\[
(\partial_{r} + \mathscr{L})(\varphi)(r,x)\geq \text{(resp. $\leq$)}\,\, f(r,x,u(r,x))
\]
for all $(r,x)\in [0,T)\times S$ and every $\varphi\in\underline{\mathscr{A}}^{L}u(r,x)$ (resp.~$\varphi\in\overline{\mathscr{A}}^{L}u(r,x)$).

To give another reasonable space of test functions, let us in this context for each $(r,x)\in [0,T)\times S$ define $\mathscr{U}_{r,x}$ to be the set of all $(\mathscr{F}_{t})_{t\in [r,T]}$-progressively measurable processes $\beta:[r,T]\times\Omega\rightarrow\mathbb{R}^{d}$ satisfying \eqref{Local Martingale Condition} and $E_{r,x}\big[M_{T}^{r,\beta}\big] = 1$. By $\underline{\mathscr{A}}\,u(r,x)$ we denote the set of all $\varphi\in C_{b}^{1,2}([0,T)\times S)$ for which there is $\tau\in\mathscr{T}$ with $\tau > r$ $P_{r,x}$-a.s. such that
\[
(u-\varphi)(r,x) \geq  E_{r,x}[M_{T}^{r,\beta}(u-\varphi)(\widetilde{\tau}\wedge\tau,X^{\widetilde{\tau}\wedge\tau})]
\] 
for all $\widetilde{\tau}\in\mathscr{T}$ with $\widetilde{\tau}\in [r,r+\delta)$, every $\beta\in\mathscr{U}_{r,x}$, and some $\delta\in (0,T-r)$. In addition, let $\overline{\mathscr{A}}\,u(r,x) := -\underline{\mathscr{A}}\,(-u)(r,x)$.

\begin{Lemma}\label{Test Function Lemma}
Let $u\in B_{b}([0,T]\times S)$ be non-anticipative, $(r,x)\in [0,T)\times S$, and $L\geq 0$, then $\underline{\mathscr{P}}\,u(r,x)\subset\underline{\mathscr{A}}\,u(r,x)\subset
\underline{\mathscr{A}}^{L}u(r,x)\subset\underline{\mathscr{S\!P}}\,u(r,x)$. In particular, each stochastic viscosity solution to \eqref{Parabolic Terminal Value Problem} is a viscosity solution in the sense of~\cite{PPDEViscosity} and every such solution is a right-hand viscosity solution.
\end{Lemma}

Of course, the second assertion of above proposition remains true if solution is either replaced by sub- or supersolution.

\section{Derivation of the results}\label{Derivation of the results}

\subsection{Mild solutions, stochastic control, and test functions}\label{Mild solutions and test functions}

First, we prove the characterization of mild solutions given by Lemma~\ref{Characterization of Mild Solutions}. To this end, note that if $u$ is a mild solution to \eqref{Parabolic Terminal Value Problem}, then $\phi:[0,T]\times S\rightarrow\mathbb{R}$ defined via
\begin{equation}\label{Mild Solution Representation}
\phi(s,y):=g(y) - \int_{s}^{T}f(t,y^{t},u(t,y^{t}))\,dt
\end{equation}
is $\mathscr{B}([0,T])\otimes\mathscr{S}_{T}$-measurable and the function $[0,T]\rightarrow\mathbb{R}$, $s\mapsto\phi(s,y)$ is continuous for each $y\in S$. Moreover, $E_{r,x}[|\phi(s,X)|] < \infty$ and $E_{r,x}[u(s,X^{s})] = E_{r,x}[\phi(s,X)]$ for all $r,s\in [0,T]$ with $r\leq s$ and each $x\in S$.

\begin{proof}[Proof of Lemma~\ref{Characterization of Mild Solutions}]
The necessity of the stated conditions follows directly from the definition of a mild solution by taking $\tau = T$. To show the sufficiency, note that $u$ is automatically non-anticipative, as $P_{r,x} = P_{r,x^{r}}$ for all $(r,x)\in [0,T]\times S$. Let $(r,x)\in [0,T]\times S$ and $\tau$ be a finite $(\mathscr{F}_{t})_{t\in [r,T]}$-stopping time. We define $\phi:[0,T]\times S\rightarrow\mathbb{R}$ via \eqref{Mild Solution Representation}, then the strong Markov property of $\hat{\mathscr{X}}$ entails that $E_{r,x}[|u(\tau,X^{\tau})|] < \infty$ and $E_{r,x}[u(\tau,X^{\tau})] = E_{r,x}[\phi(\tau,X)]$. Thus,
\begin{align*}
E_{r,x}[u(\tau,X^{\tau})] &= E_{r,x}[g(X^{T})] - E_{r,x}\bigg[\int_{\tau}^{T}f(s,X^{s},u(s,X^{s}))\,ds\bigg]\\
&= u(r,x) + E_{r,x}\bigg[\int_{r}^{\tau}f(s,X^{s},u(s,X^{s}))\,ds\bigg].
\end{align*}
\end{proof}

\begin{proof}[Proof of Proposition~\ref{control prop}]
Note that, since every mild solution to \eqref{Parabolic Terminal Value Problem} is automatically bounded, existence and uniqueness of mild solutions are a direct consequence of Theorem~\ref{mild sol existence thm}. We follow the proof of \cite[Theorem 2.8]{SchiedFuel}. In a first step, we can assume without loss of generality that $\nu_0>0$ and confine our attention to non-increasing and non-negative control processes $\nu$; this can be seen as in \cite[Lemma 4.1]{SchiedFuel}. Next, we claim that
\begin{equation}\label{u conv g eq}
\lim_{t\uparrow T} u(t,X^{t}) = g(X^{T})\quad\text{$P_{0,x}$-a.s.}
\end{equation}
Indeed, the Markov property of $\hat{\mathscr{X}}$, martingale convergence, and the continuity of the  paths of $X$ yield that
\[
\lim_{t\uparrow T}E_{t,X^t}[g(X^{T})]= \lim_{t\uparrow T}E_{0,x}[g(X^{T})|\mathscr{F}_{t}]=E_{0,x}[g(X^{T})|\sigma(\cup_{t<T}\mathscr{F}_{t})]=g(X^{T})
\]
$P_{0,x}$-a.s. Moreover, since $u$ and the function $[0,T)\times S\rightarrow\mathbb{R}$, $(r,x)\mapsto f(r,x,u(r,x))$ are bounded, we get that $\lim_{t\uparrow T}\int_{t}^{T}f(s,X^{s},u(s,X^{s}))\,ds=0$. So, \eqref{u conv g eq} follows.

Now take a control process of the form $\nu(t)=\nu_{0}+\int_{0}^{t}\dot\nu(s)\,ds$ for a progressive and integrable process $(\dot\nu(t))_{t\in [0,T]}$ and define for $t<T$,
\begin{equation}\label{Ct def eq}
C^{\nu}_{t}:=\int_{0}^{t}\Big(|\dot\nu(s)|^{p}\eta(s,X^{s})+|\nu(s)|^{p}\alpha(s,X^{s})\Big)\,ds+|\nu(t)|^{p}u(t,X^{t}).
\end{equation}
Then \eqref{u conv g eq} yields that $P_{0,x}$-a.s.,
\[
\lim_{t\uparrow T} C^{\nu}_{t} = C^{\nu}_{T}:=\int_{0}^{T}\Big(|\dot\nu(s)|^{p}\eta(s,X^{s})+|\nu(s)|^{p}\alpha(s,X^{s})\Big)\,ds+|\nu(T)|^{p}g(X^{T}).
\]
Next, the Markov property of $\hat{\mathscr{X}}$ implies that $P_{0,x}$-a.s.,
\begin{align*}
M_{t}&:=u(t,X^{t})+\int_{0}^{t}\Big(\alpha(s,X^{s})-\frac{u(s,X^{s})^{q}}{(q-1)\eta(s,X^{s})^{q-1}}\Big)\,ds\\
&=E_{0,x}\bigg[g(X^{T})+\int_{0}^{T}\Big(\alpha(s,X^{s})-\frac{u(s,X^{s})^{q}}{(q-1)\eta(s,X^{s})^{q-1}}\Big)\,ds\,\bigg|\,\mathscr{F}_{t}\bigg].
\end{align*}
Hence, $M$ is a right-continuous martingale. Applying It\^o's formula to \eqref{Ct def eq} and arguing as in the proof of \cite[Proposition 4.4]{SchiedFuel}, we now arrive at
\[
dC^{\nu}_{t}=\eta(t,X^{t})\Phi_{p}\bigg(|\dot\nu(t)|,\nu(t)\Big(\frac{u(t,X^{t})}{\eta(t,X^{t})}\Big)^{q-1}\bigg)\,dt+\nu(t)^{p}\,dM_{t},
\]
where 
\[
\Phi_{p}(y,z)=y^{p}-pyz^{p-1}+(p-1)z^{p},\quad y,z\ge0,
\]
satisfies $\Phi_{p}(y,z)\geq 0$ with equality if and only if $y=z$. We hence obtain that
\begin{align*}
J(\nu)&=E_{0,x}[C^\nu_{T}]= \nu_{0}^{p}u(0,x)+E_{0,x}\bigg[\int_{0}^{T}\eta(t,X^t)\Phi_{p}\bigg(|\dot\nu(t)|,\nu(t)\Big(\frac{u(t,X^{t})}{\eta(t,X^{t})}\Big)^{q-1}\bigg)\,dt\bigg]
\end{align*}
and that $J(\nu)=\nu_0^{p}u(0,x)$ if and only if $P_{0,x}$-a.s.
\[
|\dot\nu(t)|=\nu(t)\Big(\frac{u(t,X^{t})}{\eta(t,X^{t})}\Big)^{q-1}\quad\text{for a.e.~$t\in[0,T]$.}
\]
This implies the assertion.
\end{proof}

\begin{proof}[Proof of Lemma~\ref{Test Function Lemma}]
As the second assertion is an immediate consequence of the first, we only show the first claim. From $\mathscr{U}_{r}^{L}\subset\mathscr{U}_{r,x}$ we directly get $\underline{\mathscr{A}}\,u(r,x) $ $\subset\underline{\mathscr{A}}^{L}u(r,x)$. The inclusion $\underline{\mathscr{A}}^{0}u(r,x)$ $\subset\underline{\mathscr{S\!P}}\,u(r,x)$ follows from $M_{T}^{r,0} = 1$ $P_{r,x}$-a.s. We notice that if $L'\geq 0$ is such that $L'\leq L$, then $\mathscr{U}_{r}^{L'} \subset \mathscr{U}_{r}^{L}$, which in turn gives us that $\underline{\mathscr{A}}^{L}u(r,x)\subset\underline{\mathscr{A}}^{L'}u(r,x)$. Hence, $\underline{\mathscr{A}}^{L}u(r,x) \subset \underline{\mathscr{A}}^{0}u(r,x)\subset\underline{\mathscr{S\!P}}\,u(r,x)$.  It remains to verify that
$\underline{\mathscr{P}}\,u(r,x)\subset\underline{\mathscr{A}}\,u(r,x)$.

Thus, let $\varphi\in\underline{\mathscr{P}}\,u(r,x)$, then $(u-\varphi)(r,x)$ $\geq (u-\varphi)(s,y)$ for all $(s,y)\in [r,T)\times S$ with $d_{S}((s,y),(r,x)) < \delta$ and some $\delta\in (0,T-r)$. We define $\tau:=\inf\{t\in [r,T]\,|\, \|X^{t} -x^{r}\|\geq \delta/3\}\wedge (r + \delta/2)$, then $\tau\in\mathscr{T}$ and $\tau > r$ $P_{r,x}$-a.s. Let $\widetilde{\tau}\in\mathscr{T}$ with $\widetilde{\tau}\geq r$, then
\[
d_{S}((\widetilde{\tau}\wedge\tau, X^{\widetilde{\tau}\wedge\tau}),(r,x)) \leq d_{S}((\tau,X^{\tau}),(r,x)) \leq \delta/2 + \|X^{\tau} - x^{r}\| < \delta
\]
on $\{X^{r} = x^{r}\}$. Hence, $(u-\varphi)(r,x) \geq (u-\varphi)(\widetilde{\tau}\wedge\tau,X^{\widetilde{\tau}\wedge\tau})$ on the same set. Let $\beta\in\mathscr{U}_{r,x}$, then, as $M^{r,\beta}$ is positive and $E_{r,x}[M_{T}^{r,\beta}]=1$, we get that $(u-\varphi)(r,x)\geq E_{r,x}[M_{T}^{r,\beta}(u-\varphi)(\widetilde{\tau}\wedge\tau,X^{\widetilde{\tau}\wedge\tau})]$. Thus, $\varphi\in\underline{\mathscr{A}}\,u(r,x)$.
\end{proof}

\subsection{Proofs of the main results}\label{Proof of the main results}

Let us begin with a crucial limit inequality.

\begin{Lemma}\label{Viscosity Limit Lemma}
Let $(r,x)\in [0,T)\times S$ and $\tau$ be an $(\mathscr{F}_{t})_{t\in [r,T]}$-stopping time. Assume that $\varphi\in B([r,T)\times S)$ is non-anticipative and the following two conditions hold:
\begin{enumerate}[(i)]
\item $\int_{r}^{t\wedge\tau}|\varphi(s,X^{s})|\,ds$ is finite and $P_{r,x}$-integrable for all $t\in [r,T)$.
\item There are $\zeta\in (0,T-r)$ and $c\geq 0$ so that $|\varphi(s,X^{s})|\leq c$ for all $s\in [r,(r+\zeta)\wedge\tau]$ $P_{r,x}$-a.s.
\end{enumerate}
If $\varphi$ is upper right-hand semicontinuous at $(r,x)$, then
\[
\limsup_{t\downarrow r}E_{r,x}\bigg[\int_{r}^{t\wedge\tau}\frac{\varphi(s,X^{s})}{t-r}\,ds\bigg] \leq \varphi(r,x)P_{r,x}(\tau > r).
\]
\end{Lemma}

\begin{proof}
Let $\varepsilon > 0$ and $\omega\in\{X^{r}=x^{r}\}\cap\{\tau > r\}$. Then there exists $\delta > 0$ such that $\varphi(s,y) < \varphi(r,x) + \varepsilon$ for every $(s,y)\in [r,T)\times S$ with $d_{S}((s,y),(r,x)) < \delta$. Since $X(\omega)$ is right-continuous, there is $\gamma\in (0,T-r)$ such that $\|X^{s}(\omega) - x^{r}\| < \delta/2$ for each $s\in [r,r + \gamma)$. We define $\eta:= \gamma\wedge(\delta/2)\wedge(\tau(\omega) - r)$, then
\[
\int_{r}^{t\wedge\tau(\omega)}\frac{\varphi(s,X^{s}(\omega))}{t-r}\,ds \leq \varphi(r,x) + \varepsilon
\]
for every $t\in (r,r+\eta)$, because $t < \tau(\omega)$ and $d_{S}((s,X^{s}(\omega)),(r,x))$ $= (s-r) + \|X^{s}(\omega) - x^{r}\| < \delta$ for all $s\in [r,t]$. Thus, we have shown that
\[
\limsup_{t\downarrow r}\int_{r}^{t\wedge\tau}\frac{\varphi(s,X^{s})}{t-r}\,ds\leq \varphi(r,x) \quad P_{r,x}\text{-a.s.~on $\{\tau > r\}$.}
\]
Since $\int_{r}^{t\wedge\tau}|\varphi(s,X^{s})|\,ds\leq c(t-r)$ for each $t\in [r,r+\zeta]$ $P_{r,x}$-a.s., the claim follows from Fatou's lemma.
\end{proof}

This allows us to prove our first main result.

\begin{proof}[Proof of Theorem~\ref{Main Result 1}]
We consider the case that $u$ is a mild subsolution. Let $(r,x)\in [0,T)\times S$ and $\varphi\in\underline{\mathscr{S\!P}}\,u(r,x)$. Then there are $\delta\in (0,T-r)$ and some $(\mathscr{F}_{t})_{t\in [r,T]}$-stopping time $\tau$ with $P_{r,x}(\tau > r) > 0$ such that
\begin{equation}\label{Viscosity Maximum}
(u-\varphi)(r,x) \geq E_{r,x}[(u-\varphi)(\widetilde{\tau}\wedge\tau,X^{\widetilde{\tau}\wedge\tau})]
\end{equation}
for each $\widetilde{\tau}\in\mathscr{T}$ with $\widetilde{\tau}\in [r,r+\delta)$. As the functions $[0,T)\times S\rightarrow\mathbb{R}$, $(s,y)\mapsto (\partial_{s} +\mathscr{L})(\varphi)(s,y)$ and $[0,T)\times S\rightarrow\mathbb{R}$, $(s,y)\mapsto f(s,y,u(s,y))$ are right-continuous at $(r,x)$, they are right-hand locally bounded there. That is, there are $c\geq 0$ and $\gamma\in (0,\delta]$ such that 
\[
|(\partial_{s} +\mathscr{L})(\varphi)(s,y)|\vee|f(s,y,u(s,y))|\leq c
\]
for each $(s,y)\in [r,T)\times S$ with $d_{S}((s,y),(r,x)) < \gamma$. We set $\widetilde{\tau}:=\inf\{t\in [r,T]\,|\, \|X^{t} - x^{r}\| \geq \gamma/2\}\wedge T\in\mathscr{T}$, which gives $d_{S}((s,X^{s}(\omega)),(r,x)) < \gamma$ for all $\omega\in\{X^{r}=x^{r}\}$ and each $s\in [r,(r+\gamma/3)\wedge\widetilde{\tau}(\omega)]$. Let $\hat{\tau}:=\widetilde{\tau}\wedge\tau$, then the stopped process $[r,T)\times\Omega\rightarrow\mathbb{R}$,
\[
(t,\omega)\mapsto \varphi(t\wedge\hat{\tau}(\omega),X^{t\wedge\hat{\tau}}(\omega)) - \int_{r}^{t\wedge\hat{\tau}(\omega)}(\partial_{s} + \mathscr{L})(\varphi)(s,X^{s}(\omega))\,ds
\]
is an $(\mathscr{F}_{t})_{t\in [r,T)}$-martingale under $P_{r,x}$, as the $\mathscr{L}$-martingale property of $\mathscr{X}$ and optional sampling entail. Moreover, because $u$ is a mild subsolution to \eqref{Parabolic Terminal Value Problem}, it follows that 
\begin{align*}
E_{r,x}[(u-\varphi)(t\wedge\hat{\tau},X^{t\wedge\hat{\tau}})] &\geq (u-\varphi)(r,x) + E_{r,x}\bigg[\int_{r}^{t\wedge\hat{\tau}}f(s,X^{s},u(s,X^{s}))\,ds\bigg]\\
&\quad - E_{r,x}\bigg[\int_{r}^{t\wedge\hat{\tau}}(\partial_{s} + \mathscr{L})(\varphi)(s,X^{s})\,ds\bigg]\\
\end{align*}
for all $t\in [r,T)$. Hence, we obtain from \eqref{Viscosity Maximum} that
\[
\frac{1}{t-r} E_{r,x}\bigg[\int_{r}^{t\wedge\hat{\tau}}(\partial_{s} + \mathscr{L})(\varphi)(s,X^{s})\,ds\bigg]\geq \frac{1}{t-r}E_{r,x}\bigg[\int_{r}^{t\wedge\hat{\tau}}f(s,X^{s},u(s,X^{s}))\,ds\bigg]
\]
for each $t\in (r,r+\gamma/3)$. Because $|(\partial_{s} + \mathscr{L})(\varphi)(s,X^{s})|\vee|f(s,X^{s},u(s,X^{s}))|\leq c$ for every $s\in [r,(r+\gamma/3)\wedge\hat{\tau}]$ $P_{r,x}$-a.s., Lemma~\ref{Viscosity Limit Lemma} allows us to take the limit $t\downarrow r$, which establishes that
\[
(\partial_{r} + \mathscr{L})(\varphi)(r,x)\geq f(r,x,u(r,x)).
\]
Thus, $u$ is a stochastic viscosity subsolution to \eqref{Parabolic Terminal Value Problem}. Eventually, if $u$ is a mild supersolution, then similar arguments yield that it is also a stochastic viscosity supersolution.
\end{proof}

From now on, we let $\alpha,\beta\in B([0,T)\times S)$ be two non-anticipative functions such that $\alpha(\cdot,x)$ and $\beta(\cdot,x)$ are locally integrable for each $x\in S$ and $f(t,x,z) = \alpha(t,x) + \beta(t,x)z$ for all $(t,x,z)\in [0,T)\times S\times D$. Then we can verify another limit equality without assuming right-continuity of the mild solution in question.

\begin{Lemma}\label{Viscosity Limit Lemma 2}
Let $(r,x)\in [0,T)\times S$ and $\tau$ be an $(\mathscr{F}_{t})_{t\in [r,T]}$-stopping time. Suppose that $\beta$ is right-continuous at $(r,x)$, and there are $\zeta\in (0,T-r)$ and $c\geq 0$ such that $|\beta(s,X^{s})|\leq c$ for all $s\in [r,(r + \zeta)\wedge\tau]$ $P_{r,x}$-a.s. Then each mild solution $u$ to \eqref{Parabolic Terminal Value Problem} fulfills
\[
\lim_{t\downarrow r}E_{r,x}\bigg[\int_{r}^{t\wedge\tau} \frac{\beta(s,X^{s})}{t-r}u(s,X^{s})\,ds\bigg] = \beta(r,x)u(r,x)P_{r,x}(\tau > r).
\]
\end{Lemma}

\begin{proof}
We define $\phi:[0,T]\times S\rightarrow\mathbb{R}$ by \eqref{Mild Solution Representation}, then the Borel measurable function $[0,T)\rightarrow\mathbb{R}$, $s\mapsto\beta(s,X^{s}(\omega))\phi(s,X(\omega))$ is locally integrable for each $\omega\in\Omega$. Moreover,
\begin{equation}\label{Crucial Linear Viscosity Inequality}
\begin{split}
\int_{r}^{t\wedge\tau}|\beta(s,X^{s})\phi(s,X)|\,ds &\leq c(t-r)|g(X^{T})|\\
&\quad + c(t-r)\int_{r}^{T}|\alpha(s,X^{s}) + \beta(s,X^{s})u(s,X^{s})|\,ds
\end{split}
\end{equation}
for all $t\in [r,r+\zeta]$ $P_{r,x}$-a.s. As the right-hand expression is finite and $P_{r,x}$-integrable, it follows from Fubini's theorem and the Markov property of $\hat{\mathscr{X}}$ that
\[
E_{r,x}\bigg[\int_{r}^{t\wedge\tau}|\beta(s,X^{s})|E_{s,X^{s}}[|\phi(s,X)|]\,ds\bigg] = E_{r,x}\bigg[\int_{r}^{t\wedge\tau}|\beta(s,X^{s})\phi(s,X)|\,ds\bigg] < \infty
\]
for each $t\in [r,r+\zeta]$. Because $|u(s,X^{s})| \leq E_{s,X^{s}}[|\phi(s,X)|]$ for all $s\in [r,T]$, another application of Fubini's theorem and the Markov property of $\hat{\mathscr{X}}$ yield that
\[
E_{r,x}\bigg[\int_{r}^{t\wedge\tau}\beta(s,X^{s})u(s,X^{s})\,ds\bigg] = E_{r,x}\bigg[\int_{r}^{t\wedge\tau}\beta(s,X^{s})\phi(s,X)\,ds\bigg]
\]
for every $t\in [r,r+\zeta]$. The next step is to choose some $P_{r,x}$-null set $N\in\mathscr{F}$ such that $|\beta(s,X^{s}(\omega))|\leq c$ for all $\omega\in N^{c}$ and each $s\in [r,(r+\zeta)\wedge\tau(\omega)]$. We let $\varepsilon > 0$ and $\omega\in N^{c}\cap\{X^{r} = x^{r}\}\cap\{\tau > r\}$, then the right-continuity of $\beta$ at $(r,x)$ yields $\delta > 0$ such that
\[
|\phi(r,X(\omega))||\beta(s,y) - \beta(r,x)| < \varepsilon/2
\]
for all $(s,y)\in [r,T)\times S$ with $d_{S}((s,y),(r,x)) < \delta$. Since $X(\omega)$ and the function $[0,T]\rightarrow\mathbb{R}$, $s\mapsto\phi(s,X(\omega))$ are right-continuous, we get $\gamma\in (0,T-r)$ such that $\|X^{s}(\omega) - x^{r}\| < \delta/2$ and $c|\phi(s,X(\omega)) - \phi(r,X(\omega))| < \varepsilon/2$ for each $s\in [r,r+\gamma)$. Consequently, $|\beta(s,X^{s}(\omega))\phi(s,X(\omega)) - \beta(r,x)\phi(r,X(\omega))| < \varepsilon$ for every $s\in [r,r+\eta)$, where $\eta:=\gamma\wedge(\delta/2)\wedge\zeta\wedge(\tau(\omega) - r)$. Hence,
\[
\bigg|\int_{r}^{t\wedge\tau(\omega)}\frac{\beta(s,X^{s}(\omega))}{t-r}\phi(s,X(\omega))\,ds - \beta(r,x)\phi(r,X(\omega))\bigg| < \varepsilon
\]
for all $t\in (r,r+\eta)$. Therefore, we have proven that
\[
\lim_{t\downarrow r}\int_{r}^{t\wedge\tau}\frac{\beta(s,X^{s})}{t-r}\phi(s,X)\,ds = \beta(r,x)\phi(r,X) \quad \text{$P_{r,x}$-a.s.~on $\{\tau > r\}$.}
\]
Because $E_{r,x}[\phi(r,X)\mathbbm{1}_{\{\tau > r\}}] = E_{r,x}[E_{r,x}[\phi(r,X)|\mathscr{F}_{r}]\mathbbm{1}_{\{\tau > r\}}] = u(r,x)P_{r,x}(\tau > r)$ and \eqref{Crucial Linear Viscosity Inequality} holds, the claim follows from dominated convergence.
\end{proof}

We are now concerned with another decisive limit inequality.

\begin{Lemma}\label{Viscosity Limit Lemma 3}
Let $(r,x)\in [0,T)\times S$ and $\varphi\in B([r,T)\times S)$ be non-anticipative. Suppose that $(r_{n},x_{n})_{n\in\mathbb{N}}$ is a sequence in $[r,T)\times S$, $(t_{n})_{n\in\mathbb{N}}$ is a sequence in $[r,T)$, and $(\tau_{n})_{n\in\mathbb{N}}$ is a sequence of $(\mathscr{F}_{t})_{t\in [0,T]}$-stopping times such that the following three conditions hold:
\begin{enumerate}[(i)]
\item $\tau_{n} > r_{n}$ $P_{r_{n},x_{n}}$-a.s., $\tau_{n}\geq r_{n}$, and $r_{n} < t_{n}$ for each $n\in\mathbb{N}$. In addition, $\lim_{n\uparrow\infty} (r_{n},x_{n})=(r,x)$ and $\lim_{n\uparrow\infty} t_{n} = r$.
\item $\int_{r_{n}}^{t_{n}\wedge\tau_{n}}|\varphi(s,X^{s})|\,ds$ is finite for all $n\in\mathbb{N}$ and there exists $c\geq 0$ such that $|\varphi(s,X^{s})|\leq c$ for all $s\in [r_{n},t_{n}\wedge\tau_{n}]$ $P_{r_{n},x_{n}}$-a.s.~for every $n\in\mathbb{N}$.
\item $\lim_{n\uparrow\infty} P_{r_{n},x_{n}}(\tau_{n}\leq t_{n}) = 0$ and $\lim_{n\uparrow\infty} P_{r_{n},x_{n}}(\|X^{t_{n}} - x_{n}^{r_{n}}\|\geq \gamma) = 0$ for each $\gamma > 0$.
\end{enumerate}
If $\varphi$ is upper right-hand semicontinuous at $(r,x)$, then
\[
\limsup_{n\uparrow\infty} E_{r_{n},x_{n}}\bigg[\int_{r_{n}}^{t_{n}\wedge\tau_{n}}\frac{\varphi(s,X^{s})}{t_{n}-r_{n}}\,ds\bigg]
\leq \varphi(r,x).
\]
\end{Lemma}

\begin{proof}
Let $\varepsilon > 0$, then there is some $\delta > 0$ such that $\varphi(s,y) < \varphi(r,x) + \varepsilon/4$ for each $(s,y)\in [r,T)\times S$ with $d_{S}((s,y),(r,x)) < \delta$. By (i), we can choose $n_{0}\in\mathbb{N}$ such that
\[
(t_{n} -r_{n}) + d_{S}((r_{n},x_{n}),(r,x)) < \delta/2
\]
for all $n\in\mathbb{N}$ with $n\geq n_{0}$. Moreover, for each $n\in\mathbb{N}$ we let $Y^{(n)}:[r_{n},T]\times\Omega\rightarrow\mathbb{R}_{+}$ be given by $Y_{s}^{(n)}(\omega):=\|X^{s}(\omega)-x_{n}^{r_{n}}\|$ and set $\sigma_{n}:=\inf\{t\in [r_{n},T]\,|\, \|X^{t}-x_{n}^{r_{n}}\|\geq\delta/2\}$, then $Y^{(n)}$ is an $(\mathscr{F}_{t})_{t\in [r_{n},T]}$-adapted process with increasing continuous paths and $\sigma_{n}$ is an $(\mathscr{F}_{t})_{t\in [r_{n},T]}$-stopping time with $\sigma_{n} > r_{n}$ $P_{r_{n},x_{n}}$-a.s.~and $\{\sigma_{n} > s \} = \{Y_{s}^{(n)} < \delta/2\}$ for all $s\in [r_{n},T]$. This yields that
\[
E_{r_{n},x_{n}}\bigg[\int_{r_{n}}^{t_{n}\wedge\tau_{n}}\frac{\varphi(s,X^{s})}{t_{n}-r_{n}}\mathbbm{1}_{\big\{Y_{s}^{(n)} < \delta/2\big\}}\,ds\bigg] \leq \frac{\varphi(r,x)}{t_{n}-r_{n}}E_{r_{n},x_{n}}[(t_{n}\wedge\tau_{n}\wedge\sigma_{n} -r_{n})] + \varepsilon/4
\]
for every $n\in\mathbb{N}$ with $n\geq n_{0}$, since $d_{S}((s,X^{s}(\omega)),(r,x)) \leq (s-r_{n}) + Y_{s}^{(n)}(\omega)$ $+\, d_{S}((r_{n},x_{n}),(r,x))$ $< \delta$ for all $\omega\in\{X^{r_{n}}=x^{r_{n}}\}$ and each $s\in [r_{n},t_{n}\wedge\tau_{n}(\omega)\wedge\sigma_{n}(\omega)]$. We observe that
\[
\frac{1}{t_{n}-r_{n}}E_{r_{n},x_{n}}[(t_{n} - t_{n}\wedge\tau_{n}\wedge\sigma_{n})] \leq P_{r_{n},x_{n}}(\tau_{n}\leq t_{n}) + P_{r_{n},x_{n}}(Y_{t_{n}}^{(n)}\geq\delta/2)
\]
for each $n\in\mathbb{N}$, since $(t_{n} - t_{n}\wedge\tau_{n}\wedge\sigma_{n}) = (t_{n} - \tau_{n}\wedge\sigma_{n})\mathbbm{1}_{\{\tau_{n}\wedge\sigma_{n}\leq t_{n}\}}$ $\leq (t_{n}-r_{n})\mathbbm{1}_{\{\tau_{n}\wedge\sigma_{n}\leq t_{n}\}}$. At the same time it follows from (ii) that
\[
E_{r_{n},x_{n}}\bigg[\int_{r_{n}}^{t_{n}\wedge\tau_{n}}\frac{|\varphi(s,X^{s})|}{t_{n}-r_{n}}\mathbbm{1}_{\big\{Y_{s}^{(n)}\geq\delta/2\big\}}\,ds\bigg] \leq c P_{r_{n},x_{n}}(Y_{t_{n}}^{(n)}\geq\delta/2)
\]
for all $n\in\mathbb{N}$. For $c':=c\vee |\varphi(r,x)|$ there is $n_{1}\in\mathbb{N}$ such that $c' P_{r_{n},x_{n}}(\tau_{n}\leq t_{n}) < \varepsilon/4$ and $c' P_{r_{n},x_{n}}(Y_{t_{n}}^{(n)}\geq\delta/2) < \varepsilon/4$ for all $n\in\mathbb{N}$ with $n\geq n_{1}$, due to (iii). Hence, we set $n_{2}:=n_{0}\vee n_{1}$, then we obtain that
\begin{align*}
E_{r_{n},x_{n}}\bigg[\int_{r_{n}}^{t_{n}\wedge\tau_{n}}\frac{\varphi(s,X^{s})}{t_{n}-r_{n}}\,ds\bigg] &\leq \varphi(r,x) +  \frac{c'}{t_{n}-r_{n}}E_{r_{n},x_{n}}[(t_{n}-t_{n}\wedge\tau_{n}\wedge\sigma_{n})] + \varepsilon/2\\
&< \varphi(r,x) + \varepsilon
\end{align*}
for each $n\in\mathbb{N}$ with $n\geq n_{2}$. This entails the assertion.
\end{proof}

To clarify the way we proceed, note that for every function $u:[0,T]\times S\rightarrow\mathbb{R}$ that is right-hand locally bounded from above and each $(r,x)\in [0,T)\times S$, there is a sequence $(r_{n},x_{n})_{n\in\mathbb{N}}$ in $[r,T)\times S$ with $\lim_{n\uparrow\infty} (r_{n},x_{n})=(r,x)$ and $\lim_{n\uparrow\infty} u(r_{n},x_{n}) = u^{\leftarrow}(r,x)$. This technique is well-known in the literature of viscosity solutions (see for example Pham~\cite{PhamStochControl}[Section 4.3]). We now use the natural filtration $(\hat{\mathscr{F}}_{t})_{t\in [0,T]}$ of $X$.

\begin{Lemma}\label{Viscosity Limit Lemma 4}
Let $(r,x)\in [0,T)\times S$, $\beta$ be right-continuous at $(r,x)$, and  $u$ be a right-hand locally bounded mild solution to \eqref{Parabolic Terminal Value Problem}. Suppose that $(r_{n},x_{n})_{n\in\mathbb{N}}$ is a sequence in $[r,T)\times S$, $(t_{n})_{n\in\mathbb{N}}$ is a sequence in $[r,T)$, and $(\tau_{n})_{n\in\mathbb{N}}$ is a sequence of $(\hat{\mathscr{F}}_{t})_{t\in [0,T]}$-stopping times such that the following three conditions hold:
\begin{enumerate}[(i)]
\item $\tau_{n} > r_{n}$ $P_{r_{n},x_{n}}$-a.s., $\tau_{n}\geq r_{n}$, and $r_{n} < t_{n}$ for every $n\in\mathbb{N}$. Furthermore, $\lim_{n\uparrow\infty}(r_{n},x_{n})=(r,x)$, $\lim_{n\uparrow\infty} u(r_{n},x_{n})$ $= u^{\leftarrow}(r,x)$, and $\lim_{n\uparrow\infty} t_{n} = r$.
\item There is $c\geq 0$ such that $|\alpha(s,X^{s})|\vee|\beta(s,X^{s})|\vee|u(s,X^{s})|\leq c$ for each $s\in [r_{n},t_{n}\wedge\tau_{n}]$ $P_{r_{n},x_{n}}$-a.s.~for every $n\in\mathbb{N}$.
\item $\lim_{n\uparrow\infty} P_{r_{n},x_{n}}(\tau_{n}\leq t_{n}) = 0$ and $\lim_{n\uparrow\infty} P_{r_{n},x_{n}}(\|X^{t_{n}} - x_{n}^{r_{n}}\|\geq \gamma) = 0$ for all $\gamma > 0$.
\end{enumerate}
Then
\[
\lim_{n\uparrow\infty} E_{r_{n},x_{n}}\bigg[\int_{r_{n}}^{t_{n}\wedge\tau_{n}}\frac{\beta(s,X^{s})}{t_{n}-r_{n}}u(s,X^{s})\,ds\bigg]
= \beta(r,x)u^{\leftarrow}(r,x).
\]
\end{Lemma}

\begin{proof}
Since $u$ is a mild solution to \eqref{Parabolic Terminal Value Problem}, we obtain from (ii) that $|E_{r_{n},x_{n}}[u(t_{n}\wedge\tau_{n},X^{t_{n}\wedge\tau_{n}})] - u(r_{n},x_{n})|$ $\leq c(1+c)(t_{n}-r_{n})$ for each $n\in\mathbb{N}$. Hence,
\begin{equation}\label{Mild Solution Cauchy Limit}
\lim_{n\uparrow\infty} E_{r_{n},x_{n}}[u(t_{n}\wedge\tau_{n},X^{t_{n}\wedge\tau_{n}})] = \lim_{n\uparrow\infty} u(r_{n},x_{n}) = u^{\leftarrow}(r,x).
\end{equation}
We note that, because the function $[0,T]\times S\rightarrow\mathbb{R}_{+}$, $(s,y)\mapsto |\beta(s,y) - \beta(r,x)|$ is right-continuous at $(r,x)$, Lemma~\ref{Viscosity Limit Lemma 3} implies that
\[
\lim_{n\uparrow\infty}
E_{r_{n},x_{n}}\bigg[\int_{r_{n}}^{t_{n}\wedge\tau_{n}}\frac{|\beta(s,X^{s}) - \beta(r,x)|}{t_{n}-r_{n}}\,ds\bigg]
= 0.
\]
So, from the hypothesis that $|u(t_{n}\wedge\tau_{n},X^{t_{n}\wedge\tau_{n}})| \leq c$ $P_{r_{n},x_{n}}$-a.s.~for all $n\in\mathbb{N}$ and the fact that $\beta(r,x) = \int_{r_{n}}^{t_{n}\wedge\tau_{n}}\beta(r,x)/(t_{n}-r_{n})\,ds + \beta(r,x)(t_{n}-t_{n}\wedge\tau_{n})/(t_{n}-r_{n})$ for each $n\in\mathbb{N}$, we readily infer from \eqref{Mild Solution Cauchy Limit} and (iii) that 
\[
\lim_{n\uparrow\infty}
E_{r_{n},x_{n}}\bigg[\int_{r_{n}}^{t_{n}\wedge\tau_{n}}\frac{\beta(s,X^{s})}{t_{n}-r_{n}}u(t_{n}\wedge\tau_{n},X^{t_{n}\wedge\tau_{n}})\,ds\bigg]
= \beta(r,x)u^{\leftarrow}(r,x).
\]
Consequently, the claim follows once we have shown that
\begin{equation}\label{Viscosity Limit 4 Equation}
\lim_{n\uparrow\infty} E_{r_{n},x_{n}}\bigg[\int_{r_{n}}^{t_{n}\wedge\tau_{n}}\frac{\beta(s,X^{s})}{t_{n}-r_{n}}(u(t_{n}\wedge\tau_{n},X^{t_{n}\wedge\tau_{n}}) - u(s,X^{s}))\,ds\bigg] = 0.
\end{equation}
To this end, we let $n\in\mathbb{N}$ and set $\tau_{n,s}:=\tau_{n}\vee s$ for each $s\in [r_{n},t_{n}]$, then $\tau_{n,s}$ is an $(\hat{\mathscr{F}}_{t})_{t\in [s,T]}$-stopping time. As $u$ is a mild solution to \eqref{Parabolic Terminal Value Problem}, we get that
\[
E_{s,y}[u(t_{n}\wedge\tau_{n,s},X^{t_{n}\wedge\tau_{n,s}})] = u(s,y) + E_{s,y}\bigg[\int_{s}^{t_{n}\wedge\tau_{n,s}}\alpha(s',X^{s'}) + \beta(s',X^{s'})u(s',X^{s'})\,ds'\bigg]
\]
for each $(s,y)\in [r_{n},t_{n}]\times S$. Hence, Fubini's theorem and the strong Markov property of $\hat{\mathscr{X}}$ yield that
\begin{align*}
\bigg|&E_{r_{n},x_{n}}\bigg[\int_{r_{n}}^{t_{n}\wedge\tau_{n}}\frac{\beta(s,X^{s})}{t_{n}-r_{n}}(u(t_{n}\wedge\tau_{n},X^{t_{n}\wedge\tau_{n}}) -u(s,X^{s}))\,ds\bigg]\bigg|\\
& =\bigg|\int_{r_{n}}^{t_{n}}E_{r_{n},x_{n}}\bigg[\frac{\beta(s,X^{s})}{t_{n}-r_{n}}(E_{s,X^{s}}[u(t_{n}\wedge\tau_{n,s},X^{t_{n}\wedge\tau_{n,s}})] - u(s,X^{s}))\mathbbm{1}_{\{\tau_{n} > s\}}\bigg]\,ds\bigg|\\
&\leq c(1+c)\int_{r_{n}}^{t_{n}}E_{r_{n},x_{n}}[|\beta(s,X^{s})|\mathbbm{1}_{\{\tau_{n} > s\}}]\,ds\leq c^{2}(1+c)(t_{n}-r_{n}),
\end{align*}
since $\tau_{n} = \tau_{n,s}$ on $\{\tau_{n} > s\}$ for all $s\in [r_{n},t_{n}]$. As $n\in\mathbb{N}$ has been arbitrarily chosen, we may take the limit $n\uparrow\infty$ to obtain \eqref{Viscosity Limit 4 Equation}, which proves the assertion.
\end{proof}

\begin{proof}[Proof of Theorem~\ref{Main Result 2}]
(i) We proceed similarly as in the proof of Theorem~\ref{Main Result 1}. Let $(r,x)\in [0,T)\times S$ and $\varphi\in\underline{\mathscr{S\!P}}\,u(r,x)$. Then there exist $\delta\in (0,T-r)$ and an $(\mathscr{F}_{t})_{t\in [r,T]}$-stopping time $\tau$ with $P_{r,x}(\tau > r) > 0$ such that
\[
(u-\varphi)(r,x) \geq E_{r,x}[(u-\varphi)(\widetilde{\tau}\wedge\tau,X^{\widetilde{\tau}\wedge\tau})]
\]
for every $\widetilde{\tau}\in\mathscr{T}$ with $\widetilde{\tau}\in [r,r+\delta)$. Let $c\geq 0$ and $\gamma\in (0,\delta]$ be such that $|(\partial_{s} + \mathscr{L})(\varphi)(s,y)|\vee|\alpha(s,y)|\vee|\beta(s,y)|\leq c$ for all $(s,y)\in [r,T)\times S$ with $d_{S}((s,y),(r,x)) < \gamma$. We set $\widetilde{\tau}:=\inf\{t\in [r,T]\,|\, \|X^{t} - x^{r}\| \geq \gamma/2\}\wedge T$ and $\hat{\tau}:=\widetilde{\tau}\wedge\tau$, then it follows that
\begin{align*}
\frac{1}{t-r}E_{r,x}\bigg[\int_{r}^{t\wedge\hat{\tau}}(\partial_{s} + \mathscr{L})(\varphi)(s,X^{s})\,ds\bigg] &\geq \frac{1}{t-r}E_{r,x}\bigg[\int_{r}^{t\wedge\hat{\tau}}\alpha(s,X^{s}) + \beta(s,X^{s})u(s,X^{s})\,ds\bigg]
\end{align*}
for each $t\in (r,r+\gamma/3)$, as $\mathscr{X}$ fulfills the $\mathscr{L}$-martingale property and $u$ is a mild subsolution to \eqref{Parabolic Terminal Value Problem}. Due to Lemmas~\ref{Viscosity Limit Lemma} and~\ref{Viscosity Limit Lemma 2}, regardless of whether $u$ is right-continuous on $[0,T)\times S$, we may take the limit $t\downarrow r$ to obtain that
\[
(\partial_{r} + \mathscr{L})(\varphi)(r,x) \geq \alpha(r,x) + \beta(r,x) u(r,x).
\]
For this reason, $u$ is a stochastic viscosity subsolution to \eqref{Parabolic Terminal Value Problem}. The fact that $u$ is also a stochastic viscosity supersolution can be proven with similar reasoning.

(ii) To verify that $u^{\leftarrow}$ is a right-hand viscosity subsolution, let $(r,x)\in [0,T)\times S$ and $\varphi\in\underline{\mathscr{P}}\,u^{\leftarrow}(r,x)$. Then there is $\delta\in (0,T-r)$ such that
\begin{equation}\label{Right-hand Viscosity Maximum}
(u^{\leftarrow} - \varphi)(r,x)\geq (u^{\leftarrow} - \varphi)(s,y)
\end{equation}
for each $(s,y)\in [r,T)\times S$ fulfilling $d_{S}((s,y),(r,x)) < \delta$. Certainly, there exists a sequence $(r_{n},x_{n})_{n\in\mathbb{N}}$ in $[r,T)\times S$ such that $\lim_{n\uparrow\infty} (r_{n},x_{n})=(r,x)$ and $\lim_{n\uparrow\infty} u(r_{n},x_{n})= u^{\leftarrow}(r,x)$. We set
\[
\eta_{n}:= (u^{\leftarrow}-\varphi)(r,x) - (u-\varphi)(r_{n},x_{n})\quad\text{for all $n\in\mathbb{N}$.}
\]
Then, since $\lim_{n\uparrow\infty}\eta_{n} = 0$, there exists a sequence $(t_{n})_{n\in\mathbb{N}}$ in $[r,T)$ such that $r_{n} < t_{n}$ for each $n\in\mathbb{N}$,  $\lim_{n\uparrow\infty} t_{n} = r$, and $\lim_{n\uparrow\infty} \eta_{n}/(t_{n}-r_{n}) = 0$. For instance, we could have set $t_{n} := r_{n} + (1/2)\min\big\{\sqrt{|\eta_{n}|}, T- r_{n}\big\}$ for each $n\in\mathbb{N}$.

Let us choose $c > 0$ and $\gamma\in (0,\delta]$ such that $|(\partial_{s} + \mathscr{L})(\varphi)(s,y)|\vee|\alpha(s,y)|\vee|\beta(s,y)|\leq c$ for all $(s,y)\in [r,T)\times S$ with $d_{S}((s,y),(r,x)) < \gamma$. We set 
\[
\tau_{n}:=\inf\{t\in [r_{n},T]\,|\, \|X^{t} - x_{n}^{r_{n}}\|\geq \gamma/2\}\quad\text{for each $n\in\mathbb{N}$,}
\]
then $\tau_{n}$ is an $(\hat{\mathscr{F}}_{t})_{t\in [r_{n},T]}$-stopping time with $\tau_{n} > r_{n}$ $P_{r_{n},x_{n}}$-a.s. In addition, let $n_{0}\in\mathbb{N}$ be such that $(t_{n} -r_{n}) + d_{S}((r_{n},x_{n}),(r,x)) < \gamma/2$ for all $n\in\mathbb{N}$ with $n\geq n_{0}$. Then from \eqref{Right-hand Viscosity Maximum} and $u^{\leftarrow}\geq u$ we infer that
\[
(u^{\leftarrow} - \varphi)(r,x)\geq E_{r_{n},x_{n}}[(u-\varphi)(t_{n}\wedge\tau_{n},X^{t_{n}\wedge\tau_{n}})]
\]
for every $n\in\mathbb{N}$ such that $n\geq n_{0}$, because $d_{S}((t_{n}\wedge\tau_{n},X^{t_{n}\wedge\tau_{n}}),(r,x)) < \gamma$ on $\{X^{r_{n}} = x_{n}^{r_{n}}\}$. Moreover, since $u$ is a mild subsolution to \eqref{Parabolic Terminal Value Problem} and the stopped process $[r_{n},T)\times\Omega\rightarrow\mathbb{R}$, 
\[
(t,\omega)\mapsto\varphi(t\wedge\tau_{n}(\omega),X^{t\wedge\tau_{n}}(\omega)) - \int_{r_{n}}^{t\wedge\tau_{n}(\omega)}(\partial_{s}+\mathscr{L})(\varphi)(s,X^{s}(\omega))\,ds
\]
is an $(\mathscr{F}_{t})_{t\in [r_{n},T)}$-martingale under $P_{r_{n},x_{n}}$, it follows that
\begin{align*}
E_{r_{n},x_{n}}[(u-\varphi)(t_{n}\wedge\tau_{n},X^{t_{n}\wedge\tau_{n}})] &\geq (u-\varphi)(r_{n},x_{n})\\
& + E_{r_{n},x_{n}}\bigg[\int_{r_{n}}^{t_{n}\wedge\tau_{n}}\alpha(s,X^{s}) +\beta(s,X^{s})u(s,X^{s})\,ds\bigg]\\
&\quad - E_{r_{n},x_{n}}\bigg[\int_{r_{n}}^{t_{n}\wedge\tau_{n}}(\partial_{s} + \mathscr{L})(\varphi)(s,X^{s})\,ds\bigg]
\end{align*}
for each $n\in\mathbb{N}$ with $n\geq n_{0}$. By recalling the definition of $\eta_{n}$, this implies that
\begin{align*}
\frac{\eta_{n}}{t_{n} - r_{n}} + \frac{1}{t_{n}-r_{n}}E_{r_{n},x_{n}}\bigg[&\int_{r_{n}}^{t_{n}\wedge\tau_{n}}(\partial_{s} + \mathscr{L})(\varphi)(s,X^{s})\,ds\bigg]\\
&\geq \frac{1}{t_{n}-r_{n}} E_{r_{n},x_{n}}\bigg[\int_{r_{n}}^{t_{n}\wedge\tau_{n}}\alpha(s,X^{s}) + \beta(s,X^{s})u(s,X^{s})\,ds\bigg]
\end{align*}
for all $n\in\mathbb{N}$ with $n\geq n_{0}$. Hence, as $\{\tau_{n}\leq t_{n}\} = \{\|X^{t_{n}} - x_{n}^{r_{n}}\|\geq \gamma/2\}$ for all $n\in\mathbb{N}$ and $\lim_{n\uparrow\infty} P_{r_{n},x_{n}}(\|X^{t_{n}} - x_{n}^{r_{n}}\|\geq \gamma/2) = 0$, Lemmas~\ref{Viscosity Limit Lemma 3} and~\ref{Viscosity Limit Lemma 4} ensure that we may take the limit $n\uparrow\infty$, which yields that
\[
(\partial_{r} + \mathscr{L})(\varphi)(r,x)\geq \alpha(r,x) + \beta(r,x)u^{\leftarrow}(r,x).
\]
This shows that $u^{\leftarrow}$ is a right-hand viscosity subsolution to \eqref{Parabolic Terminal Value Problem}. Since the verification that $u_{\leftarrow}$ is a right-hand viscosity supersolution can be handled in much the same way, the claim is proven.
\end{proof}

\subsection{Measurability}

For the lemma below, let $I$ be a non-degenerate interval in $[0,T]$ and $R\in\widetilde{\mathscr{S}}_{T}$ be non-empty and stable under stopping, i.e., $x^{t}\in R$ for each $(t,x)\in [0,T]\times R$. Additionally, we introduce the filtration $(\mathscr{R}_{t})_{t\in [0,T]}$ by $\mathscr{R}_{t}:=R\cap\widetilde{\mathscr{S}}_{t}$ for all $t\in [0,T]$. This general setting allows for a simultaneous treatment of the main cases $I= [0,T)$, $I = [0,T]$, $R= \widetilde{S}$, and $R=S$.

\begin{Lemma}\label{Measurability Lemma}
Assume that $E$ is a Polish space with Borel $\sigma$-field $\mathscr{B}$ and complete metric $\varrho$ that induces the topology of $E$, and let $u:I\times R\rightarrow E$.
\begin{enumerate}[(i)]
\item The map $u$ is non-anticipative and product measurable if and only if it is progressively measurable. In this case, it is also Borel measurable.

\item Let $u$ be right-continuous relative to $d_{\infty}$, then $u(\cdot,x)$ is right-continuous for all $x\in R$. Moreover, if $R$ is closed with respect to $\|\cdot\|$ and $u$ is continuous relative to $d_{\infty}$, then $u(\cdot,x)$ is c\`{a}dl\`{a}g and left-continuous at each continuity point of $x$.

\item If $u$ is right-continuous (with respect to $d_{S}$), then it is progressively measurable.
\end{enumerate}
\end{Lemma}

\begin{proof}
(i) For only if note that, as the process $[0,T]\times\widetilde{S}\rightarrow\widetilde{S}$, $(t,x)\mapsto x^{t}$, which is the path process of $\widetilde{\xi}$, has c\`{a}dl\`{a}g paths, it is progressively measurable with respect to its natural filtration $(\widetilde{\mathscr{S}}_{t})_{t\in [0,T]}$. For this reason, the time-space process
\[
[0,T]\times\widetilde{S}\rightarrow [0,T]\times\widetilde{S},\quad (t,x)\mapsto (t,x^{t})
\]
is $(\widetilde{\mathscr{S}}_{t})_{t\in [0,T]}$-progressively measurable provided the image is endowed with the product $\sigma$-field $\mathscr{B}([0,T])\otimes\widetilde{\mathscr{S}}_{T}$. Since $u$ is non-anticipative, $u(t,x) = u(t,x^{t})$ for all $(t,x)\in I\times R$. Thus, the $(\mathscr{R})_{t\in I}$-progressive measurability of $u$ follows from its $\mathscr{B}(I)\otimes\mathscr{R}_{T}$-measurability.

For if we recall that $u$ must be product measurable and $(\mathscr{R}_{t})_{t\in I}$-adapted. As $E$ is Polish, for given $t\in I$ there is a Borel measurable map $\phi:R\rightarrow E$ such that $u(t,x)=\phi(x^{t})$ for all $x\in R$. From $u(t,x) = \phi(x^{t}) =\phi((x^{t})^{t}) = u(t,x^{t})$ for all $x\in R$ we see that $u$ is non-anticipative, as desired.

We turn to the last claim. So, let $u$ be non-anticipative and product measurable. By the definition of $d_{S}$, the map $[0,T]\times\widetilde{S}\rightarrow [0,T]\times\widetilde{S}$, $(t,x)\mapsto (t,x^{t})$ is uniformly continuous provided the domain is equipped with $d_{S}$ and the image is equipped with a product metric. In particular, Borel measurability follows. That is,
\[
\{(t,x)\in [0,T]\times\widetilde{S}\,|\, (t,x^{t})\in F\}\in\mathscr{B}([0,T]\times\widetilde{S})\quad \text{for all $F\in\mathscr{B}([0,T])\otimes\widetilde{\mathscr{S}}_{T}$}.
\]
Hence, $u^{-1}(B) = \{(t,x)\in I\times R\,|\,(t,x^{t})\in u^{-1}(B)\}\in\mathscr{B}(I\times R)$ for each $B\in\mathscr{B}$, since $u$ is non-anticipative and $u^{-1}(B)\in\mathscr{B}(I)\otimes\mathscr{R}_{T}$. Thus, $u$ is Borel measurable.

(ii) We fix $x\in R$ and let $r\in I$ with $r < \sup I$ and $\varepsilon > 0$. Then there is $\delta > 0$ such that $\varrho(u(t,y),u(r,x)) < \varepsilon$ for all $(t,y)\in I\times R$ with $t\geq r$ and $d_{\infty}((t,y),(r,x)) < \delta$. Since $x$ is right-continuous at $r$, there is $\gamma\in (0,T-r)$ such that
\[
|x(s) - x(r)| < \delta/2
\]
for all $s\in [r,r+\gamma)$. We set $\eta:=\gamma\wedge(\delta/2)$, then $\varrho(u(t,x),u(r,x)) < \varepsilon$ for all $t\in [r,r+\eta)\cap I$, because $d_{\infty}((t,x),(r,x)) < \delta/2 + \sup_{s\in [r,t]}|x(s) - x(r)| \leq \delta$.

Now let $R$ be closed with respect to $\|\cdot\|$ and $u$ be continuous relative to $d_{\infty}$. We fix $t\in I$ with $t > \inf I$ and prove that $\lim_{r\uparrow t} u(r,x) = u(t,x_{t})$, where $x_{t}\in R$ is defined via $x_{t}(s):=x(s)\mathbbm{1}_{[0,t)}(s) + x(t-)\mathbbm{1}_{[t,T]}(s)$ and satisfies $\lim_{r\uparrow t} \|x^{r}-x_{t}\| = 0$. So, let $\varepsilon > 0$, then there is $\delta > 0$ such that $\varrho(u(r,y),u(t,x_{t})) < \varepsilon$ for all $(r,y)\in I\times R$ with $d_{\infty}((r,y),(t,x_{t})) < \delta$. Moreover, choose $\gamma\in (0,t)$ such that
\[
|x(s) - x(t-)| < \delta/4
\]
for all $s\in (t-\gamma,t)$, We define $\eta:=\gamma\wedge(\delta/2)$, then $\varrho(u(r,x),u(t,x_{t})) < \varepsilon$ for every $r\in (t-\eta,t)\cap I$, since $d_{\infty}((r,x),(t,x_{t})) < \delta/2 + |x(r) - x(t-)| + \sup_{s\in [r,t)}|x(t-) -x(s)| < \delta$. This in fact concludes the proof, because $x_{t} = x^{t}$ whenever $x$ is continuous at $t$.

(iii) As $u$ must be right-continuous with respect to $d_{\infty}$, it follows from (ii) that $u(\cdot,x)$ is right-continuous for each $x\in R$. Moreover, $u$ is $(\mathscr{R}_{t})_{t\in I}$-adapted. Indeed, let $t\in I$, then we easily see that $R_{t}:=\{x\in R\,|\,x = x^{t}\}$ is closed in $R$ with respect to $\rho$ and the restriction $u_{t}$ of $u(t,\cdot)$ to $R_{t}$ is continuous with respect to $\rho$. Non-anticipation yields that $u(t,x) = u_{t}(x^{t})$ for all $x\in R$, which in turn guarantees the $\mathscr{R}_{t}$-measurability of $u(t,\cdot)$. Now the claim follows.
\end{proof}

%%%%%%%%%%%%%%%%%%%%%%%%%%%%%%%%%%%%%%%

\bibliography{library}
\bibliographystyle{plain}
\end{document}